\documentclass[12pt,reqno]{amsart}

\usepackage[colorlinks,linkcolor=blue,citecolor=blue]{hyperref}
\usepackage[abs]{overpic}
\usepackage[normalem]{ulem} 

\usepackage{epigraph}

\usepackage{amsmath, amscd, amssymb, amsthm,amsfonts}
\usepackage{euscript, mathrsfs, latexsym,mathabx,stmaryrd}
\usepackage{color}
\usepackage{hyperref} 

\usepackage{lmodern}
\usepackage[T1]{fontenc} 

\usepackage{multicol}

\ifx\pdfoutput\undefined
\usepackage{graphicx}
\else
\fi

\ifx\xetex
\usepackage{fontspec}
\usefont{EU1}{lmr}{m}{n}
\else
\fi

\usepackage{url} 
\usepackage{tikz}
\usetikzlibrary{matrix,arrows,positioning}
\tikzset{node distance=2cm, auto}

\usepackage[nodayofweek]{datetime}

\parskip=\medskipamount

\newcommand{\conj}[1]{\quad\textnormal{ #1 }\quad}

\newcommand{\inp}[1]{\ensuremath{\langle #1 \rangle}}

\newcommand{\module}[0]{\operatorname{-mod}}

\newcommand{\rmodule}[0]{\operatorname{mod-}\!}

\newcommand{\normaltext}[1]{\textnormal{#1}}

\makeatletter
\def\imod#1{\allowbreak\mkern2.5mu({\operator@font mod}\,#1)}
\makeatother

\renewcommand{\a}{\alpha}
\renewcommand{\b}{\beta}
\newcommand{\e}{\epsilon}
\newcommand{\vp}{\varphi}
\newcommand{\opp}{\oplus}
\newcommand{\ott}{\otimes}

\renewcommand{\l}{\lambda}
\newcommand{\s}{\sigma}

\newcommand{\Ga}{\Gamma}
\newcommand{\ga}{\gamma}

\renewcommand{\d}{\delta}

\newcommand{\aC}{\mathcal{C}}

\newcommand{\aF}{\mathcal{F}}

\newcommand{\aT}{\mathcal{T}}

\newcommand{\bA}{\mathbf{A}}

\newcommand{\bn}{\mathbf{n}}

\newcommand{\fH}{\mathfrak{H}}

\usepackage{bbm}

\newcommand{\FF}{\mathbb{F}}

\newcommand{\PP}{\mathbb{P}}
\newcommand{\QQ}{\mathbb{Q}}
\newcommand{\RR}{\mathbb{R}}

\newcommand{\ZZ}{\mathbb{Z}}

\newcommand{\kk}{\mathbb{k}}

\newcommand{\eC}{\EuScript{C}}
\newcommand{\eD}{\EuScript{D}}

\theoremstyle{plain}

\newtheorem{thm}{Theorem}[section]

\newtheorem{theorem}[thm]{Theorem}

\newtheorem{conjecture}[thm]{Conjecture}

\newtheorem{prop}[thm]{Proposition}

\newtheorem{cor}[thm]{Corollary}
\newtheorem{lemma}[thm]{Lemma}

\theoremstyle{remark}

\theoremstyle{definition}

\newtheorem{example}[thm]{Example}

\newtheorem{defn}[thm]{Definition}

\newtheorem{definition}[thm]{Definition}

\newtheorem{notation}[thm]{Notation}
\newtheorem{rmk}[thm]{Remark}

\numberwithin{equation}{section}

\DeclareMathAlphabet{\mathscrbf}{OMS}{mdugm}{b}{n}

\makeatletter
\def\imod#1{\allowbreak\mkern2.5mu({\operator@font mod}\,#1)}
\makeatother

\makeatletter
\def\namedlabel#1#2{\begingroup
   \def\@currentlabel{#2}%
   \label{#1}\endgroup
}
\makeatother

\newcommand{\nt}[1]{\normaltext{\enskip #1\enskip}}

\usepackage{relsize}
\newcommand{\aprod}{\varoast}

\newcommand{\Ainf}{A_\infty}

\newcommand{\h}{h}

\newcommand{\A}{\Lambda} 
\renewcommand{\bA}{\bar{\Lambda}} 

\renewcommand{\kk}{k}

\newcommand{\lab}[1]{\normaltext{#1}}

\newcommand{\vnp}[1]{\lvert #1 \rvert}

\newcommand{\xto}[1]{\xrightarrow{#1}}
\newcommand{\xfrom}[1]{\xleftarrow{#1}}
\newcommand{\from}{\leftarrow}
\usepackage{extarrows}

\newcommand{\op}{\normaltext{op}}

\theoremstyle{plain}

\newcommand{\NAlg}{\bn}

\newcommand{\Vect}{Vect}

\newcommand{\im}{im\,}

\newcommand{\End}{End}

\newcommand{\Aut}{Aut}

\newcommand{\Obj}{Ob}
\newcommand{\Ob}{\Obj}

\renewcommand{\s}{\sigma}

\newcommand{\Hom}{Hom}
\newcommand{\Ext}{Ext}

\newcommand{\Alg}{{\bf f}}

\newcommand{\DHa}{\mathcal{D}\!\mathcal{H}}
\newcommand{\DC}{DC}
\newcommand{\tCHa}{DC}
\newcommand{\tDHa}{D\!H\!a}

\newcommand{\Rep}{Rep}

\newcommand{\E}{E}
\renewcommand{\e}{\epsilon}

\newcommand{\F}{\aF}

\begin{document}

\title[Hall algebras of surfaces]{The Hall algebras of Annuli}
\author[Benjamin Cooper]{Benjamin Cooper}
\author[Peter Samuelson]{Peter Samuelson}
\address{University of Iowa, Department of Mathematics, 14 MacLean Hall, Iowa City, IA 52242-1419 USA}
\email{ben-cooper\char 64 uiowa.edu}
\address{University of California, Riverside, Department of Mathematics, 900 University Ave., Riverside 92521 USA}
\email{psamuels\char 64 ucr.edu}

\begin{abstract}
We refine and prove the central conjecture of our first paper for annuli with at least two marked intervals on each boundary component by computing 
the derived Hall algebras of their Fukaya categories. 
\end{abstract}
\vspace{.5in}

\maketitle

\setlength{\epigraphwidth}{0.35\textwidth}

\setcounter{tocdepth}{1}
\setcounter{secnumdepth}{2}
\vspace{.5in}
\tableofcontents

\section{Introduction}\label{introsec}

The Hall algebraic approach to quantum groups developed by Ringel
\cite{Ringel} and Lusztig \cite{Lusztig} endowed the theory introduced by
Drinfeld and Jimbo with a geometric and higher categorical intepretation.
The relationship between quantum groups and Witten-Reshetikhin-Turaev
topological field theories inspired the famous conjectures of Crane and
Frenkel \cite{CraneFrenkel}. When understood as a 4-dimensional topological
field theory, the objects naturally associated to surfaces are
manifestations of skein algebras.  Morton and Samuelson
\cite{MortonSamuelson} suggested that there should be a relationship between
the Hall algebra of the Fukaya category and the skein algebra of the
surface. The present paper is one of a series which develops these ideas.

In \cite{sillyman1}, the authors introduced a series of conjectures which
suggest a Hall algebraic approach to skein algebras of surfaces. First, the
graded HOMFLY-PT skein algebra of a surface was conjectured to be isomorphic
to the Hall algebra of the Fukaya category of that surface, at least once
the former was rigorously defined.
$$H_q(S) \xto{\sim} \tDHa (D^\pi\F(S))$$ 
Second, the authors gave a conjectural generators and relations presentation of the 
Hall algebra for most surfaces, which was called the Naive Conjecture, see
\S\ref{naiveconjsec}.  This conjecture implies that if the surface has
sufficiently many marked intervals on its the boundary then the 
Hall algebra can be presented in an explicit way using a decomposition of the surface into disks.

 The main result of the present paper occurs in \S\ref{nannulisec} where we use quiver
techniques to compute the Hall algebras of the Fukaya categories of annuli
$K_{m,n}$ with $m$ marked intervals on one boundary component and $n$ marked
intervals on the other when $m\geq 2$ and $n\geq 2$. 
Roughly stated, this proof has two steps. First, we show that the
Hall algebra is isomorphic to the composition subalgebra of the
derived Hall algebra of a certain quiver. This quiver composition subalgebra has a
presentation by a result of Hernandez and Leclerc. Second, we show that the
relations in this presentation correspond to triangles in the surface, which
then hold in the naive algebra by functoriality.

For example, when the surface $S$ is the annulus $K_{2,2}$ with two
marked intervals on each boundary component, there is an arc system $A$
consisting of two boundary arcs intervals $E_1$, $F_1$ and $E_3, F_3$ between
each of the two marked intervals and two internal arcs $E_2=F_4$ and
$E_4=F_2$. This is pictured below.
\vspace{.1in}
\begin{center}\label{fig:annuluspresex}
\begin{overpic}[scale=1]
{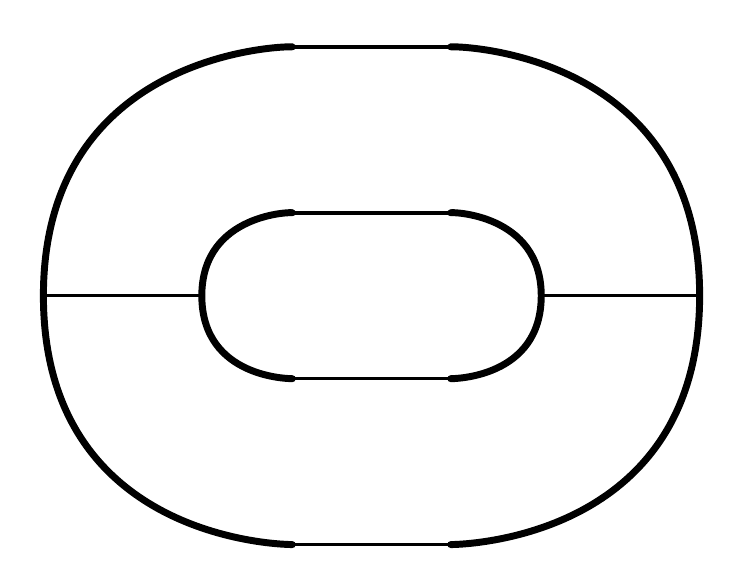}
\put(102,160){$E_1$}
\put(102,0.5){$F_1$}

\put(102,95.5){$E_3$}
\put(102,65.5){$F_3$}

\put(31,88){$E_4$}
\put(31,73){$F_2$}

\put(173,88){$E_2$}
\put(173,73){$F_4$}
\end{overpic}
\end{center}
\vspace{.1in}
Cutting the annulus along the arc system produces two disks $(D^2_1,\A_4)$ and $(D^2_2,\A'_4)$ with four arcs $\A_4 = \{E_1,E_2,E_3,E_4\}$ and $\A'_4 = \{F_1,F_2,F_3,F_4\}$. Since there are enough marked intervals, the Naive Conjecture implies that the composition subalgebra $\Alg(S,A)$ 
is generated by suspensions of the arcs $E_{i,n} = \s^nE_i$ and $F_{j,m}= \s^mF_j$ for $1\leq i,j\leq 4$ and $n,m\in\ZZ$ subject to the \eqref{g1:itm} and \eqref{g3:itm} relations:
\begin{align*}
E_{2,n}=F_{4,n}  \conj{ and } E_{4,n} = F_{2,n} & \textnormal{ for all } n \in \ZZ,\\
[E_{2,n},F_{2,m}]_1 = 0  \conj{ and }  [F_{2,n},E_{2,m}]_1 = 0 & \textnormal{ for all } n,m\in\ZZ,
\end{align*}
as well as the relations which hold within each disk algebra $\Alg(D^2_1,\A_4)$ and $\Alg(D^2_2,\A'_4)$.

Our results show that the Naive Conjecture holds for annuli with at least 2
marked intervals on each boundary component, and 
this gives further evidence to suggest the correctness of this conjecture. 
Unfortunately, quiver
methods do not extend directly to the general case which will be addressed
in \cite{sillyman2} using different techniques.

Our study of the Hall algebra of these annuli is mirror to the study of the
Hall algebra of coherent sheaves on weighted projective lines
\cite{BSWPL}. More complicated surfaces obtained by the gluing of
annuli are mirror to sheaves on gluings of weighted projective lines
\cite{LP}. Following \cite{MortonSamuelson}, we expect the Hall algebra of
the torus to parallel the Hall algebra of sheaves on the elliptic
curve \cite{BS, SV11}.

While this paper was under review, F. Haiden posted the preprints
\cite{Haiden1, Haiden2} which provide a geometric approach to some of the
questions studied here.

\subsection*{Acknowledgments} 
The author would like to thank K. Kawamuro, F. Haiden, Y. Lekili and O. Schiffmann for
their collegiality and helpful conversations.  The first author would also like to
thank the Max Planck Institute for the hospitality and excellent working
conditions during July and August 2018. Most of this paper was written during his
visit.

\section{Recollections and definitions}\label{recanddefsec}
In \cite{sillyman1}, we showed that finitary marked surfaces $(S,M)$ have
functorially defined derived Hall algebras $\tDHa(D^\pi\F(S,A))$ and
composition subalgebras $\Alg(S,A)$ associated to each arc system $A$.  The
purpose of this section is both to recall a few of these details and to
establish notation.  New materials include the relationship between marked
surfaces and ribbon graphs.  The foliation data from \cite[\S
4.3]{sillyman1} is extended to all surfaces. Balanced foliations are
introduced here, as are fully formal arc systems, in
Def. \ref{fullyformaldef}, which are used later in \S\ref{nanparamsec}.

\subsection{Surface topology}\label{surfsec}
When a surface $S$ has corners, the smooth part of its boundary
$\partial S = \partial_0 S \sqcup \partial_1 S$ is a disjoint union of
closed $1$-manifolds, $\partial_0 S$, and open intervals, $\partial_1 S$.  A
{\em marked surface} is an oriented surface $S$ with corners together with a
subset $M\subset\partial S$ of the boundary which contains each closed
component, $\partial_0S \subset M$, and every other component of
$\partial_1 S$.  A marked surface is {\em finitary} when it is compact, has
boundary and there is at least one marked interval on each boundary
component.

An {\em arc} in a marked surface $S$ is a closed embedded interval which
intersects $M$ transversely at its endpoints and is not isotopic to an
interval in $M$.  Arcs are considered up to ambient isotopies in which endpoints may
move within respective components of $M$.  A {\em boundary arc} is an arc
which is isotopic to the closure of a component of
$\partial S \backslash M$. An {\em internal arc} is an arc which is not a
boundary arc.  An {\em arc system} $A$ in $S$ is a collection of pairwise
disjoint non-isotopic arcs. A {\em full arc system} is an arc system $A$
containing all boundary arcs that cuts $S$ into a collection of disks.  

\begin{defn}\label{fullyformaldef}
An arc system $A$ is {\em fully formal} when it cuts $S$ into a collection of
disks each of which contains exactly one boundary arc that is {\em not}
contained in $A$, see \cite[\S 3]{HKK}.
\end{defn}

\begin{rmk}\label{ffrmk}
A fully formal arc system is not a full arc system. The missing boundary arc
can be recovered as a twisted complex of the other arcs in the disk which
contains it.
  \end{rmk}

\begin{example}\label{exex}
If $(S,M)$ is a disk $D^2$ with $m$ marked intervals then there
is a minimal full arc system $\A_m$ consisting only of boundary arcs.
Removing any one of the arcs $X\in \A$ produces an arc
system $\A\backslash \{X\}$ which is fully formal. 
\end{example}

A {\em map $f : (S,M) \to (T,N)$ of marked surfaces} is an orientation preserving
immersion which satisfies $f(M)\subset N$ and maps the boundary arcs of $S$
to disjoint non-isotopic arcs in $T$.  Note that such maps are not necessarily closed under composition. 
If $f$ also takes arcs in an arc system for $S$ to arcs in an arc system for $T$
then $f$ induces a strict $A_\infty$-functor
between associated Fukaya categories. For more details see \cite[\S 3]{HKK}.

\subsubsection{Marked surfaces as graphs}\label{ribbonsec}
\newcommand{\Ed}{E(\Ga)}
\newcommand{\Vr}{V(\Ga)}

\begin{defn}\label{graphdef}
Fix a countably infinite set $\Omega$. A {\em graph} $\Ga$ is a subset $\Ga \subset \Omega$ of {\em half-edges} equipped two partitions
  $$\Ga = \sqcup_{e\in \Ed} e \conj{ and } \Ga = \sqcup_{v\in \Vr} v$$
one into {\em edges}  $e\in \Ed$ and one into {\em vertices} $v\in \Vr$. Each edge $e\subset \Ga$ is required to have cardinality one or two. 
If $\vnp{e} = 2$ then $e$ is called an {\em internal edge} and if $\vnp{e}=1$ then $e$ is called a {\em boundary edge}. The size $\vnp{v}$ of a vertex  $v\subset\Ga$ is called its {\em valence} and it is required to be greater than two; $\vnp{v} \geq 3$.
A {\em ribbon graph} is a graph in which each vertex is equipped with a cyclic order.
\end{defn}

\begin{prop}\label{trivprop}
  Marked surfaces $(S,M)$ equipped with full arc systems $A$ are in one-to-one correspondence with ribbon graphs.
  \end{prop}

  This correspondence associates a vertex to each disk in $S \backslash A$,
  an edge to each arc which is in the boundary of two disks, and a half-edge
  to each arc which is in the boundary of a unique disk. The cyclic ordering
  of half-edges in a vertex is determined by the orientation of the surface.

\newcommand{\gr}{\fH}
\subsubsection{Gradings of surfaces and graphs}\label{gradingsec}
A {\em grading} of an oriented surface $(S,M)$ is a foliation by lines; i.e. a
section $\eta\in \Ga(S,\PP(TS))$ 
of the projectivization of the tangent bundle. Two homotopic foliations give equivalent gradings.  An
oriented marked surface together with a homotopy class of foliation or
equivalent data is said to be {\em graded}.  The set of {\rm homotopy classes of
gradings on $S$} is 
\begin{equation}\label{greqn}
\gr(S) := \pi_0 \Ga(S, \PP(TS)).
\end{equation}

A {\em grading} on an immersed curve $\ga : I \to S$ in a graded surface $(S,\eta)$ is a smoothly varying choice of
path 
$\tilde{\ga}(p) : \dot{\ga}(p) \to \eta_p$ from the line defined by
the derivative of $\gamma$ to the foliation at each point $p=\ga(t)$.  If $\ga_1$ and
$\ga_2$ are two such graded curves and intersect at a point $p$ then the
intersection index is given by
$$i_p(\ga_1,\ga_2) := \tilde{\ga}_1(p) \cdot \kappa \cdot \tilde{\ga}_2(p)^{-1} \in \pi_1(\PP(T_pS),\eta_p) \cong \ZZ,$$
where $\kappa$ is the shortest counterclockwise path from $\dot{\ga}_1$ to $\dot{\ga}_2$.
As maps in the Fukaya category, points of intersection between graded curves are graded by intersection
index. The identification $\pi_1(\PP(T_pS),\eta_p) \cong\ZZ$ is determined by the orientation of $S$.
If $\ga$ is a graded curve then the {\em $n$-fold suspension} of $\ga$ is $\s^n\ga = \ga[-n]$ is the graded curve given by multiplying each path generator by $n$-times the generator of $\pi_1(\PP(T_pS))$.

Up to homotopy, the choice of $\eta$ is determined by a collection of
integers assigned to the half-edges of the graph associated to the surface by Prop. \ref{trivprop}.

\begin{defn}\label{foldatadef}
A function $f : \Ga \to \ZZ$ on the half-edges of $\Ga$
is {\em foliation data} when  $\sum_{h\in v} f(h) = \vnp{v} - 2$ for every vertex $v \in V(\Ga)$.
\end{defn}

It is useful to know how foliation data transforms under operations on
ribbon graphs.  Suppose $f$ is foliation data on a ribbon graph $\Ga$ and
the graph $\Ga/e$ is obtained by {\em collapsing an internal edge}. This operation is dual to deleting an arc that separates two disks in a disk decomposition of the surface. If
$e = \{h_i,h'_j\}$ where $h_i \in H(v) = \{h_1, h_2, \ldots, h_n\}$ and $h'_j \in
H(v') = \{h'_1, h'_2, \ldots, h'_m\}$ in $\Ga$ then this new graph is determined by the assignments  $\Ga/e := \Ga\backslash e$, $E(\Ga/e) := E(\Ga)\backslash e$, $H(v\#v') = (H(v)\cup H(v'))\backslash \{ h_i, h'_j \}$, see \cite[\S 1.1]{Igusa}. When this is so, the map $f$ determines foliation data $f_{/e}$ by the assignment
\begin{equation}\label{gluefoleq}
  f_{/e}(h) := \left\{\begin{array}{ll} f(h_{i-1})+f(h'_j) & \normaltext{ if } h = h_{i-1} \\ f(h'_{j-1}) + f(h_i) & \normaltext { if } h = h'_{j-1} \\ f(h) & \normaltext{ otherwise }  \end{array} \right.
  \end{equation}
for  $h \in H(v\#v')$ in $\Ga/e$. Otherwise $f_{/e} := f$ on $(\Ga/e)\backslash H(v\# v')$.

\subsection{Fukaya categories}\label{fukcatsec}
\newcommand{\EF}{\F}
\newcommand{\DEF}{D^\pi\F}

\begin{defn}{$(\F(S,A))$}\label{fukcatdef}
If $S$ is an oriented graded marked surface and $A$ is an arc system then \cite{HKK} defines an $\Ainf$-category $\F(S,A)$ with objects $\Ob(\F(S,A))=A$ given by the set of graded arcs in $A$. The morphisms in $\F(S,A)$ are $k$-linear combinations of boundary paths. 

Given two distinct arcs $X$ and $Y$ in $A$, a {\em boundary path} from $X$ to $Y$ is a non-constant path which is contained in a marked interval and which starts on an endpoint of $X$, follows the reverse orientation of the boundary 
and ends on an endpoint $Y$. This means that a boundary path has the surface to its right.
When $X$ and $Y$ coincide, the trivial path $1_X$ is considered a boundary path. The {\em degree} of a boundary path $\gamma : [0,1] \to S$ from $X$ to $Y$ is given by
$$\vnp \ga := i_{\ga(0)}(X,\ga) - i_{\ga(1)}(Y,\ga)$$
for any grading of $\ga$. 

The $\Ainf$-structure on $\F(S,A)$ is defined below.
\begin{description}
\item[$(\mu_1)$] The map $\mu_1$ is always zero.
\item[$(\mu_2)$] The map $\mu_2$ is given by concatenation of boundary paths: if $a$ and $b$ can be concatenated 
then 
$$\mu_2(b,a) := (-1)^{\vnp a} a\cdot b,$$ 
otherwise, $\mu_2(b,a) := 0$.
\item[$(\mu_m)$] Suppose that $(D^2,\A)$ is a disk with $m$ marked intervals and $m \geq 3$. Let $\A$ be the boundary arcs and $\{c_1,\ldots,c_m\}$ the boundary paths between them ordered cyclically according to the reverse disk orientation. Then a {\em disk sequence} is a collection of boundary paths 
$\{f\circ c_1,\ldots,f \circ c_m\}$ in $(S,A)$ for some map $f : (D^2,\A) \to (S,A)$ of marked surfaces. 

If $a_1,\ldots,a_m$ is a disk sequence and $b$ is a boundary path then 
$$\mu_m(a_m,\ldots,a_1\cdot b) := (-1)^{\vnp b} b \conj{ or } \mu_m(b\cdot a_m,\ldots, a_1) := b$$
when $a_1\cdot b\ne 0$ or $b\cdot a_m\ne 0$ respectively; otherwise the map $\mu_m$ is defined to be zero.
\end{description}
\end{defn}

Notice that the definition above only depends on the underlying ribbon graph
together with a choice of foliation data $f : \Ga \to \ZZ$. This foliation data can be used to describe the degrees of maps in the Fukaya category. If $\a_i : h_i \to h_{i+1}$ is a boundary path from the edge dual to $h_i$ to the edge dual to $h_{i+1}$, then the intersection index determines a degree
$$\vnp{\a_i} := i_p(h_i,\a_i) - i_q(h_{i+1},\a_i)$$
where $\a_i$ is an arbitrarily graded path from $p$ on the arc associated to $h_i$ to $q$ on the arc associated to $h_{i+1}$. This degree agrees with the value of the foliation data; in the sense that $f(h_i) = \vnp{\a_i}$.

\begin{defn}\label{splitcldef}
If $\aC$ is an $\Ainf$-category then there is a triangulated category
$D^\pi\aC$ called the {\em split-closed derived category}. 
It is uniquely characterized by the property of being the 
smallest split-closed triangulated subcategory of the homotopy category of $\aC$-modules $H^0(\aC\module)$
containing the image of the Yoneda embedding.
For other details, see \cite[\S 4.2.2]{sillyman1} and references therein. 
\end{defn}

\begin{defn}\label{scfukcatdef}
The {\em split-closed derived Fukaya category} $D^\pi\F(S,A)$ is the split-closed derived category of the Fukaya category $\F(S,A)$ introduced above.
  \end{defn}

\subsection{Balanced gradings} \label{balancedsec}
A balanced foliation is one which will allow the formation of graded closed curves. This brief section introduces balanced foliations and establishes a few of their elementary properties.

\begin{defn}\label{balanceddef}
A {\em double pair} consists of two arcs $X$ and $Y$ in a graded marked surface $(S,M)$ and two distinct non-trivial boundary paths $\ga, \ga' : X \to Y$ from $X$ to $Y$. A double pair is {\em balanced} when the degrees of the paths $\vnp{\ga} = \vnp{\ga'}$ agree.
A foliation $\eta$ on a surface $S$ is {\em balanced} if all double pairs are balanced.
\end{defn}

\begin{rmk}
The arcs $S$ and $T$ in the annuli $K_{m,n}$ of Fig. \ref{fig:standardarcs} are a balanced pair.
  \end{rmk}

There is a different formulation of this definition. A foliation $[\eta]\in \gr(S)$ gives a section $\eta\in [\eta]$, so that $\eta : (S, \partial S)\to (\PP(TS), \partial\PP(TS))$ which gives a class $\eta_*[S,\partial S]\in H_2(\PP(TS),\partial \PP(TS))$.  Poincar\'{e} duality determines a class $D\eta_*[S,\partial S]\in H^1(\PP(TS))$, so the Kronecker pairing gives a map
\begin{equation}\label{asteq}
\ast : \gr(S) \times H_1(\PP(TS)) \to \ZZ \conj{ where } \eta \ast [\ga] := \inp{D\eta_*[S,\partial S], [\ga]}.
\end{equation}
In particular, any immersed curve $\ga : S^1 \to S$ in $S$ gives a section $s_\ga = (\ga, [\dot{\ga}]) : S^1 \to \PP(TS)$ and so a class $[\ga] = (s_\ga)_*[S^1] \in H_1(\PP(TS))$. 

The pairing Eqn. \eqref{asteq} allows us to associate an integer to each
immersed curve, and this number has an equivalent, geometric, interpretation as a winding number or
the algebraic intersection number between the fixed field $\eta_*[S,\partial S]$ and the class $(s_\ga)_*[S^1]$. This is the number of times the tangent vector
of $\ga$ winds around the foliation as one traces through the curve. In more
detail, suppose $S_t$ defines a path from $[\ga'(t_0)]$ to
$\eta_{\ga(t_0)}$. This path can be extended in an interval
$(t_0-\e,t_0+\e)$, and then to all of $\ga$ using a finite cover. So there is
the initial path $S_t$ and the path $T_t$ from $\ga'(t_0)$ to
$\eta_{\ga(t_0)}$ obtained from traversing $\ga$ once. The composition
$ST^{-1}$ is a path from $\eta_{\ga(t_0)}$ to itself, and the pairing becomes
$$\eta\ast\ga := ST^{-1} \in \pi_1(\PP(TS),\ga(t_0)) \cong \ZZ$$
where the identification with $\ZZ$ is again determined by the orientation of $S$. 

Since $S$ has boundary, $TS$ is trivial so the fiber bundle $\RR P^1 \to \PP(TS) \to S$ gives a non-canonically split short exact sequence
$$0\from H^1(\RR P^1) \xfrom{i^*} H^1(\PP(TS)) \xfrom{p^*} H^1(S) \from 0.$$
Since a class in $H^1(\PP(TS))$ is dual to $\eta_*[S,\partial S]$ when its algebraic intersection with the vertical fiber is $1$, $i^*(D\eta_*[S,\partial S])([\RR P^1]) = 1$, the set of homotopy classes of foliations can be identified with a subset $\gr(S) = (i^*)^{-1}(V)$ where $V \in H^1(\RR P^1)$ satisfies $\inp{V, [\RR P^1]} = 1$.  The map $p^*$ gives $\gr(S)$ the structure of an $H^1(S)$-torsor, if $c\in H^1(S)$ then $[\eta + c] = [\eta] + p^*c$, see \cite[\S 1.1]{Lekili2}.

The algebraic condition in Def. \ref{balanceddef} is equivalent to the geometric condition in the proposition below.

\begin{prop}
A foliation $\eta$ is balanced if and only if $\eta \ast \ga = 0$ for all
unobstructed embedded closed curves $\ga$ in $S$.
\end{prop}
\begin{proof}
It is only necessary to consider embedded curves because the pairing factors through homology.
Suppose $\eta \ast \ga = 0$ for all closed curves $\ga$. Up to homotopy, we
can push the line field defining $\eta$ up so that it is perpendicular to
both $X$ and $Y$. If $\a, \b : X\to Y$ are two distinct non-trivial paths
between internal arcs $X$ and $Y$ then concatenating them with $X$ and $Y$
forms a loop $\ga := X\cdot \a \cdot Y\cdot \b^{-1}$. The condition that the
$\vnp{\a} = \vnp{\b}$ is equivalent to the condition that $\eta \ast \ga = 0$.  The proof of the converse is the same.
\end{proof}

\begin{cor}\label{balancedprop}
There is one balanced foliation on the annulus.
  \end{cor}
\begin{proof}
  Let $S := S^1\times [0,1]$. Since $S^1$ is a Lie group and $[0,1]$ is contractible, the tangent bundle of the annulus is trivial. This implies $\PP(TS) \cong S\times \PP(\RR^2) \cong S\times S^1$. Then a section $\eta : S \to \PP(TS)$ is a map $S\to S^1$ and, up to homotopy, such maps are in bijection with maps $S^1\to S^1$, which are classified (up to homotopy) by degree. Now under this identification, $\deg(\eta) = \eta \ast [S^1\times \{0\}]$ so that the only balanced foliation is the one which corresponds to the degree zero self-map of the circle.

\end{proof}

\subsection{Hall algebras}\label{hallsec}
As a vector space, the derived Hall algebra of a triangulated category $\aT$ satisfying certain finiteness conditions is given by $\QQ$-linear combinations of isomorphism classes of objects
$$\DHa(\aT):=\QQ\inp{\Ob(\aT)/iso}.$$
If $X,Y\in \Ob(\aT)/iso$ then their product is given by
$$XY := \sum_{L} F^{L}_{X,Y} L$$
Here $F^L_{X,Y}$ are To\"{e}n's structure constants
\begin{equation}\label{toeneq}
 F^L_{X,Y} := \frac{\vnp{\Hom(X,L)_Y}}{\vnp{\Aut(X)}}\cdot \frac{\prod_{n > 0} \vnp{\Ext^{-n}(X,L)}^{(-1)^n}}{\prod_{n > 0} \vnp{\Ext^{-n}(X,X)}^{(-1)^n}} 
\end{equation}  
for $X,Y,L \in \Ob(\aT)/iso$ where $\Hom(X,L)_Y := \{ f\in \Hom(X,L) : C(f)\cong Y \}$.

When the marked surface $(S,M)$ is finitary, Prop. 4.10 in \cite{sillyman1}
shows that there is a meaningful definition of the derived Hall algebra of the Fukaya
category $\aT = \DEF(S,A)$ for any full arc system $A$ when the coefficient field is a finite field $\FF_q$.

The {\em Euler form} is the map $\inp{\cdot,\cdot} : K_0(\aT) \ott K_0(\aT) \to \ZZ$ given by
\begin{equation}\label{eulerformeq}
  \inp{X,Y} := \sum_{n\in\ZZ} (-1)^n \dim \Ext^n(X,Y).
\end{equation}  
It can be symmetrized by setting $X\cdot Y := \inp{X,Y} + \inp{Y,X}$. Sometimes the notation $(X,Y) := X\cdot Y$ is used.
This defines a symmetric bilinear form on the Grothendieck group $K_0(\aT)$ which can be identified with the Cartan matrix for the derived category of modules over an acyclic quiver.

The {\em twisted derived Hall algebra} $\tDHa(\aT) := \DHa(\aT)\ott_{\QQ} \QQ(q)$ has 
the same basis, with multiplication twisted by the Euler form
$$X*Y := q^{\inp{Y,X}} XY.$$
(See the value of $q$ in \S\ref{kcglossarysec}.)

\subsection{Subalgebras of Hall algebras}\label{subalgsec}

If $(S,M)$ is a finitary surface endowed with a full arc system $A$ then the derived Hall algebra $\tDHa(\DEF(S,A))$ of the Fukaya category is too big.
The composition subalgebras introduced below are much more reasonable objects to study.

\begin{defn}\label{fcompalgdef}
The {\em composition subalgebra} $\Alg(S,A)$ is the image of the canonical map
$$\kappa : F(\ZZ A) \to \tDHa(\DEF(S,A))$$
from the free $\QQ(q)$-algebra on the set $\ZZ A$ of suspensions of arcs to the derived Hall algebra.
\end{defn}

This composition subalgebra is an analog of the composition subalgebra of the Hall algebra of quiver representations, which is the subalgebra generated by simple modules. This construction satisfies the following functoriality property, see \cite[Cor. 5.11]{sillyman1}.

\begin{thm}{(Embedding)}\label{embthm}
Suppose that $(S,M)$ and $(S',M')$ are marked surfaces equipped with full arc systems $A$ and $A'$. If $f : (S,M) \to (S',M')$ is an embedding which induces an injection $\pi_0(M) \to \pi_0(M')$ between sets of components of marked intervals then $f$ induces a monomorphism $f_* : \Alg(S,A) \to \Alg(S',A')$.
\end{thm}

\subsection{Disk presentations and gluing statements}\label{gluingconjecturesec}
In this subsection we first recall the presentation of the composition subalgebra $\Alg(D^2,
\A_3)$ associated to the disk with three boundary arcs from \cite{sillyman1}. Second
the algebra associated to disk formed by a gluing a family of such disks $(D^2,
\A_3)$ is recalled. For more detail and discussion see \cite{sillyman1}.

Theorem 5.16 of \cite{sillyman1} introduces a presentation for the composition subalgebra $\Alg(D^2, \A_m)$ of a disk $(D^2,m)$ with $m$ boundary arcs equipped with the minimal full arc system $\A_m$. (It is a minimal arc system because there are no internal arcs.) An important special case occurs when $m=3$. 
\vspace{.1in}
\begin{center}\label{fig:triangledisk}
\begin{overpic}[scale=.8]
{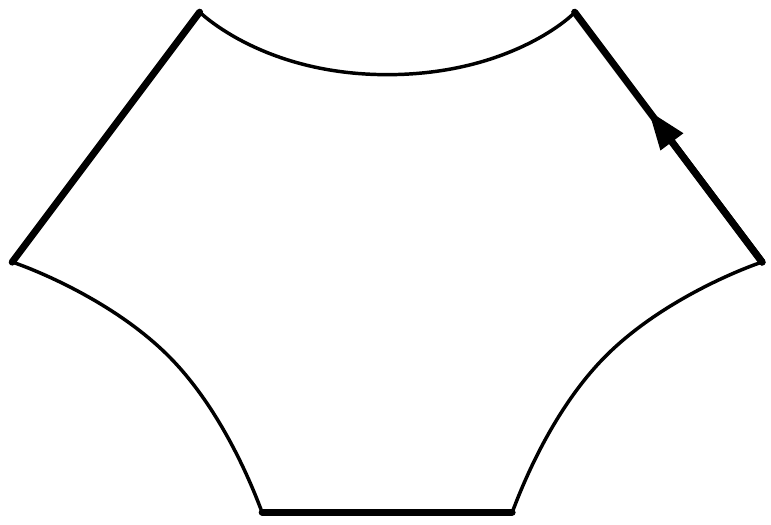}
\put(85,113){$\E_1$}
\put(17,25){$\E_3$}
\put(150,25){$\E_2$}
\put(163,90){$\a_1$}
\put(7,90){$\a_3$}
\put(83,-10){$\a_2$}
\end{overpic}
\end{center}
\vspace{.1in}

In the illustration above, the arc system consists of the three boundary arcs $\A_3 := \{\E_1,\E_2,\E_3\}$. As the boundary of the disk is oriented counterclockwise, there are boundary arcs $\a_i : \E_i\to \E_{i+1}$ for $i\in \ZZ/3$. The foliation data $h : \A_3 \to \ZZ$ satisfies $h(\E_i) = \vnp{\a_i}$. 

In the special case $m=3$, the presentation of $\Alg(D^2,\A_3)$ becomes the corollary below.

\begin{cor}\label{diskprescor}
Suppose that the disk $(D^2,3)$ with three marked intervals is equipped with a minimal arc system $\A_3 = \{\E_1,\E_2,\E_3\}$ and foliation data $\h : \A_3 \to \ZZ$.  Then the algebra $\Alg(D^2,\A_3)$ has only two families of relations, 
\begin{description}
  \item[(R1)] Self-extension:
\begin{align*}
[\E_{i,0},\E_{i,k}]_{q^{2(-1)^k}} &=\d_{k,1}\frac{q^{-1}}{q^2-1}\conj{ for } k \geq 1,
\end{align*}
\item[(R2)] Adjacent commutativity and convolution:
\begin{align*}
[\E_{i+1,k},\E_{i,h(i)}]_{q^{(-1)^{k+1}}} & = \d_{k,1} \E_{i+2,1-h(i+1)} &\conj{ for } k \geq 1 \notag,\\
[\E_{i+1,k},\E_{i,h(i)}]_{q^{(-1)^{k}}} &= 0 & \conj{ for } k<1\notag.
\end{align*}
\end{description}
\end{cor}

Due to the balancing condition in \S\ref{gradingsec} this paper will often assume that the foliation data $h$ takes a simplified form.

\begin{rmk}\label{keyrelrmk}
When $h(\E_1) = 0$, $h(\E_2) = 0$ and $h(\E_3) = 1$, the key (R2)-relations, when $k=1$ above, can be written as 
$$\E_1 = [\s \E_3, \E_2]_q,\quad\quad \E_2 = [\E_1, \E_3]_q \conj{ and } \E_3 = [\E_2, \s^{-1}\E_1]_q.$$
\end{rmk}

 The second result from \cite{sillyman1} is to do with the gluing of disks. 
 Suppose that $(D^2_1,\A_n)$ and $(D^2_2,\A_m)$ are two disks with minimal
  arc systems $\A_n = \{E_k\}_{k\in\ZZ/n}$ and $\A_m=\{F_k\}_{k\in\ZZ/m}$. Up to homotopy, the foliations of each disk can be chosen to be tangent to boundary arcs. So if $E_i\in\A_n$ and $F_j\in\A_m$ then there is a gluing defined by the quotient
$$(D^2_1,\A_n)\sqcup_{i,j} (D^2_2,\A_m) := (D^2_1 \sqcup_{i,j} D^2_2, \A_n \sqcup_{i,j} \A_m).$$
The glued disk $D^2_1 \sqcup_{i,j} D^2_2 = (D^2_1 \sqcup D^2_2) / (E_i\sim F_j)$ is formed
by the quotient which identifies the two boundary arcs on each disk. 
This
disk supports a minimal arc system
$\A_n\sqcup_{i,j}\A_m = (\A_n\backslash \{E_i\}) \sqcup
(\A_m\backslash\{F_j\})$.

\vspace{.15in}
\begin{center}\label{fig:tmpgluedisk}
\begin{overpic}[scale=0.7]
{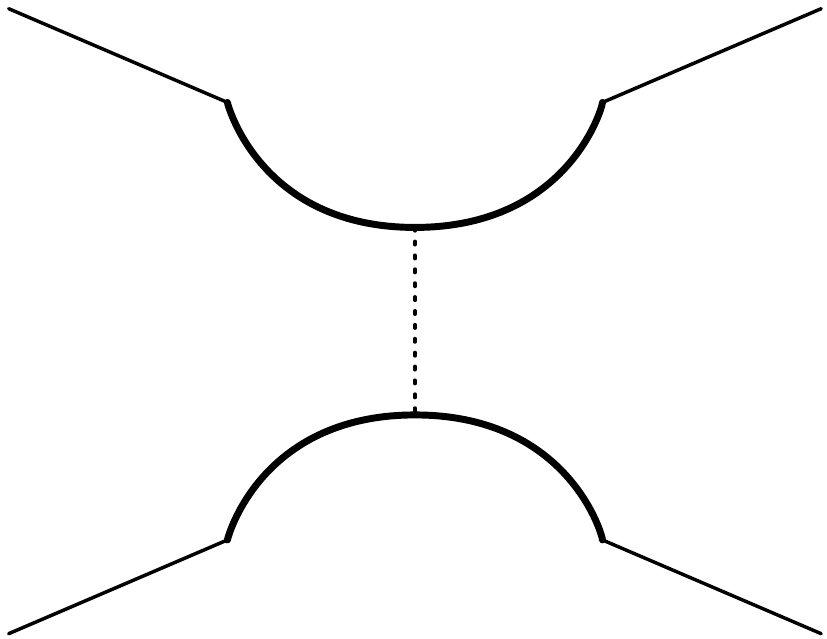}
\put(65,60){$E_i$}
\put(89,60){$F_j$}
\put(10,130){$E_{i-1}=G_{n-1}$}
\put(10,-5){$E_{i+1}=G_1$}
\put(110,130){$F_{j+1}=G_n$}
\put(110,-5){$F_{j-1}=G_{0}$}
\end{overpic}
\end{center}
\vspace{.15in}

If this minimal arc system is labelled by $\A_n\sqcup_{i,j}\A_m = \{G_k\}_{k\in\ZZ/(n+m-2)}$ then the foliation data is determined by Eqn. \eqref{gluefoleq}.

The statement below summarizes Theorem 5.23 in \cite{sillyman1} which shows how to glue disks algebraically.

\begin{thm}{(Gluing)}\label{gluethm}
The Hall algebra of the disk obtained by a gluing is isomorphic to the $\aprod$-product of Hall algebras of the disks. In more detail, there are isomorphisms
\begin{equation}\label{glueeq}
  \a : \Alg(D^2_1 \sqcup_{i,j} D^2_2, \A_n \sqcup_{i,j} \A_m) \rightleftarrows \Alg(D^2_1,\A_n)\aprod_{i,j} \Alg(D^2_2,\A_m) : \b,
  \end{equation}
where the algebra on the righthand side is given by the free product of the algebras $\Alg(D^2_1,\A_n)$ and $\Alg(D^2_2,\A_m)$ subject to the relations listed below.
\begin{description}
\item[(G1)\namedlabel{g1:itm}{\lab{G1}}] Gluing: 
$$F_{j,s} = E_{i,s}\conj{ for all } s\in\ZZ,$$
\item[(G3)\namedlabel{g3:itm}{\lab{G3}}] Far-commutativity: 
$$[E_{k,s}, F_{\ell,t}]_1 = 0 \conj{ for all } (k,\ell)\not\in Int \normaltext{ and } s, t\in\ZZ,$$
where $Int =\{ (i+1,j-1), (i,j-1), (i,j+1), (i-1,j), (i-1,j+1), (i+1,j)\}$.
\end{description}

\end{thm}  

(The set $Int$ above is the set of pairs of boundary arcs which intersect the same marked interval in the glued disk.)

\begin{rmk}\label{diskrmk}
The theorem produces a natural presentation for the algebra
associated to a decomposition of a disk $\Alg(D^2, A)$ into a family of disks which have
been glued together. This implies that the inclusion $\A_n \subset A$ induces an isomorphism of algebras
$$\Alg(D^2,\A_n) \xto{\sim} \Alg(D^2, A).$$
So the presentation $\Alg(D^2,\A_n)$ in \cite[\S 5.2.1]{sillyman1} is implied by Thm. \ref{gluethm} and the presentation for $(D^2, \A_3)$ in Cor. \ref{diskprescor}.
\end{rmk}

It is interesting to ask whether the gluing theorem above extends to
surfaces which are more general than a disk. In the next section, we recall
the naive gluing conjecture of \cite{sillyman1} which suggests circumstances
in which this construction produces the correct result.

\section{The Naive Gluing conjecture}\label{naivesec}
Suppose that $(S,M)$ is a graded marked surface with an arc system $A$ such that the
inclusions of disks $(D_i, \A_{m_i})\hookrightarrow (S,A)$ in an arc
decomposition $S\backslash A = \sqcup_i (D_i,\A_{m_i})$ determine
monomorphisms $\Alg(D_i, \A_{m_i}) \hookrightarrow \Alg(S,A)$. Then it was
conjectured in \cite[\S 5.5]{sillyman1} that the gluing relations in Theorem
\ref{gluethm} yield a presentation for the composition subalgebra
$\Alg(S,A)$. This is the Naive Conjecture \ref{naiveconj}.

In this section this conjecture is reviewed and shown to satisfy several
naturality properties with respect to the topology of surfaces.
Thm. \ref{naiveeqvthm} shows that the naive conjecture commutes with certain Pachner
moves between arc systems and Prop. \ref{nembthm} shows that the naive
conjecture commutes with the embedding theorem.

\subsection{The conjecture}\label{naiveconjsec}

This section is organized as follows. First Def. \ref{naivedef} recalls the
extension of the gluing construction in Thm. \ref{gluethm} to all
surfaces. Thm. \ref{prop:naivemaps} recalls the comparison map $\ga_A$
associated to an arc system $A$. The Naive Conjecture \ref{naiveconj} states
that the map $\ga_A$ is an isomorphism when the marked surface $(S, M)$ with
arc system $A$ has enough marked intervals as in Def. \ref{def:enough}.

\begin{definition}\label{naivedef}
Suppose that $(S,M)$ is a graded marked surface and $A$ is a full arc system. By definition, the internal arcs of $A$ cut the surface into a collection of disks
$$(S, A) = (D^2_1,\A_{m_1}) \sqcup_{i_1,j_1} (D^2_2,\A_{m_2}) \sqcup_{i_2,j_2} \cdots \sqcup_{i_{\ell-1},j_{\ell-1}} (D^2_\ell,\A_{m_\ell}).$$ 

Since the surface $(S,A)$ can be reassembled by
gluing the $\ell$ disks together, identifying the arc $i_k$ in
$\A_{m_k}$ with the arc $j_k$ in $\A_{m_{k+1}}$, the gluing
theorem 
suggests the following definition of the {\em naive algebra} $\NAlg(S,A)$
\begin{equation}\label{npeqn}
  \NAlg(S,A): = \Alg(D^2,\A_{m_1}) \aprod_{i_1,j_1} \Alg(D^2_2,\A_{m_2}) \aprod_{i_2,j_2} \cdots \aprod_{i_{\ell-1},j_{\ell-1}} \Alg(D^2_\ell,\A_{m_\ell}).
  \end{equation}
Roughly speaking, $\NAlg(S,A)$ is the free product of the disk algebras subject 
to the \eqref{g1:itm} and \eqref{g3:itm} relations from Thm. \ref{gluethm}.
More precisely, Eqn. \eqref{npeqn} is to be interpreted as a quotient of the free product of the disk algebras $\Alg(D^2_k,\A_{m_k})$ by two relations
\begin{description}
\item[\eqref{g1:itm}] If two arcs $E\in \A_{m_k}$ and $F \in \A_{m_{k'}}$ contained in distinct disks are identified by the gluing then they are identified in the naive algebra $\NAlg(S,A)$.
$$E=F$$

\item[\eqref{g3:itm}] Suppose that $E\in \A_{m_k}$ and $F\in \A_{m_{k'}}$ are contained in distinct disks. Then if the end points of $E$ and $F$ are not contained in some marked interval of $(S,M)$ then they are required to commute in the naive algebra $\NAlg(S,A)$.
$$[E,F]_1 = 0$$
  \end{description}
All relations are required to be closed under the action of suspension: if $R\in I$ is an element in the ideal $I$ generated by the relations above then $\s^nR \in I$ for all $n\in \ZZ$.
\end{definition}

The condition that $(S, M)$ has enough marked intervals is the key premise
of the naive conjecture below.

\begin{definition}\label{def:enough}
The marked surface and full arc system $(S, M, A)$ above is said to have {\em enough marked intervals} when $S\backslash A = \sqcup_{k=1}^N (D^2_k, M_k)$ and, for each $k$, the inclusion $\iota_k : (D^2_k,M_k)\to (S,M)$  is injective on connected components.
\begin{equation}\label{eq:embcond}
(\iota_k)_*: \pi_0(M_k) \to \pi_0(M) \textrm{ is injective}
\end{equation}
\end{definition}

When this criteria is satisfied, there is a comparison map $\ga_A$ from the
naive algebra to the composition subalgebra which is determined by mapping
the arcs in each disk $a\in \A_{m_i}$ to their representatives $a\in A$. In
\cite[Thm. 5.33]{sillyman1} 
it is shown that $\ga_A$ exists and is surjective.

\begin{thm}\label{prop:naivemaps}
Suppose that  $(S,M)$ is a graded marked surface with a full arc system $A$. If $(S,M,A)$ has enough marked intervals then there is a surjective map 
$$\ga_A : \NAlg(S,A) \twoheadrightarrow \Alg(S,A)\conj{ where } \ga_A(a) := a \nt{ for all } a\in A.$$
\end{thm}

The following naive gluing conjecture \cite[Conj. 5.34]{sillyman1} is a refinement of the theorem above.

\begin{conjecture}\label{naiveconj}
If $(S, M, A)$ is a graded marked surface with enough marked intervals then the comparison map $\ga_A$ is an isomorphism.
  \end{conjecture}

In \S\ref{nannulisec}, Thm. \ref{naiveisothm} shows that this conjecture holds for annuli $K_{m, n}$ when $m,n \geq 2$.

\subsection{Naturality and embedding properties}\label{arcmovesec}

The remainder of this section contains a few structural statements which
support the conjecture and are useful for proving relations in the naive algebra. 
The first goal is to show that if $\ga_A$ is an
isomorphism for some arc system $A$ then $\ga_{A'}$ exists and is also an isomorphism
for any other arc system $A'$ which has enough marked intervals and is related to $A$ by a Pachner move.
This is accomplished by studying how
Def. \ref{naivedef} transforms under Pachner moves. After this
Prop. \ref{nembthm} shows that the embedding maps of Thm. \ref{embthm} commute
with the comparison maps of Thm. \ref{prop:naivemaps}.

The theorem below shows that the naive algebras which differ by some Pachner
moves are naturally isomorphic and that this isomorphism commutes with the comparison maps $\ga$.

\begin{thm}\label{naiveeqvthm}
Suppose that $(S,M)$ is a marked surface, $A$ and $A'$ are two full arc systems such that $A' = A\backslash \{r\}$ for some internal arc $r\in A$ and both $(S,M,A)$ and $(S,M,A')$ have enough marked intervals. Then  there is an isomorphism
$$\NAlg(S,A) \cong \NAlg(S,A')$$
which makes the diagram below commute.
\begin{center}
\begin{tikzpicture}[scale=10, node distance=2cm]
\node (A) {$\NAlg(S,A)$};
\node (B) [right of=A] {$\NAlg(S,A')$};
\draw[->] (A) to node {$\sim$} (B);
\node (Ab) [below of=A] {$\Alg(S,A)$};
\node (Bb) [below of=B] {$\Alg(S,A')$};
\draw[->] (Ab) to node {$\sim$} (Bb);
\draw[->] (A) to node [swap] {$\ga_A$} (Ab);
\draw[->] (B) to node {$\ga_{A'}$} (Bb);
\end{tikzpicture} 
\end{center}
\end{thm}

\begin{proof}

First some notation. Since the arc systems $A$ and $A'$ differ by a single arc $r\in A$,
there is a unique disk $(D_r,\A_{m_r})$ in the decomposition
$S\backslash A' = \sqcup_{i=1}^N (D_i, \A_{m_i})$ containing the arc $r$, so that
$$S\backslash A = ((D_{r_1}, \A_{r_1}) \sqcup_{r,r} (D_{r_2}, \A_{r_2})) \sqcup_{i\ne r} (D_i, \A_{m_i}).$$

For the bottom row of the commutative diagram, the inclusion $A'\subset A$ determines a monomorphism $b : \Alg(S,A') \to \Alg(S,A)$ by Thm. \ref{embthm}. This map is uniquely determined by equivariance under suspension and the assignments,
\begin{equation}\label{bbeqn}
 b(x) := x \conj{ for all } x\in A'
\end{equation}
on generating arcs. Now by assumption, the disk $(D_r, \A_{m_r})$ containing $r$ as an internal arc satisfies the embedding criteria of Thm. \ref{embthm}. So there are monomorphisms 
$$\Alg(D_r, \A_{m_r}) \hookrightarrow \Alg(S,A') \xto{b} \Alg(S,A).$$
The Gluing Theorem \ref{gluethm} implies that any internal arc $r$ in $D_r$ can be expressed in terms of suspensions of the boundary arcs $\A_r$, $r=R\in \Alg(D_r, \A_{m_r})$, see Rmk. \ref{diskrmk}. So the map $b$ is also onto because $b(R) = r$. Therefore, $b$ is an isomorphism.

The rest of the theorem is an algebraic version of the geometric argument above. There are mutually inverse isomorphisms
  \begin{equation}\label{abeqn}
    \a : \NAlg(S,A) \rightleftarrows \NAlg(S,A') : \b
    \end{equation}
which commute with the maps $\ga_A$ and $\ga_{A'}$ as shown above.

For a map $\b : \NAlg(S,A') \to \NAlg(S,A)$ to commute with
$\ga_A$ and $\ga_{A'}$ it must be defined on generators by the same assignments as Eqn. \eqref{bbeqn}.
\begin{equation}\label{bdefeqn}
  \b(x) := x \conj{ for all } x \in A'.
  \end{equation}
In the same way, the map $\a : \NAlg(S,A)\to \NAlg(S,A')$ is determined by the assignments that the inverse
 $a : \Alg(S,A) \to \Alg(S,A')$ of the map $b$ makes to generators. (This follows from the fact that any $M-1$ boundary arcs generate the Fukaya-Hall algebra of the disk with $M$ boundary arcs.)

The rest of the proof shows that the assignments of $\b$ and $\a$ on generators above determine isomorphisms between naive algebras in Eqn. \eqref{abeqn}.

When the number $N$ of disks in the decomposition of $S$ by $A'$ is one, it suffices to observe that this statement reduces to the gluing theorem \ref{gluethm} because by construction $\NAlg(S,A') = \Alg(D_r,A_{m_r})$ and  $\NAlg(S,A) = \Alg(D_{r_1}, \A_{r_1}) \aprod_{r,r} \Alg(D_{r_2}, \A_{r_2})$.

When there is more than one disk, the isomorphism from the gluing theorem is shown to extend along the construction of the naive algebras. In each case, $\NAlg(S,A)$ and $\NAlg(S,A')$ is a quotient of a free product of either side of the gluing isomorphism with disk algebras subject to the gluing relations from Def. \ref{naivedef}. Assembling these quotients of free products gives the diagram below.
\begin{equation*}\begin{tikzpicture}[scale=10, node distance=2.5cm]
\node (S) {$\NAlg(S,A)$};
\node (A) [right=3cm of S] {};

\node (A1) [above=1cm of A]{$\cdots\ast \Alg(D_{r-1},\A_{m_{r-1}})\big) \ast \big(\Alg(D_r, \A_r)\big) \ast \big(\Alg(D_{r+1}, \A_{m_{r+1}}) \ast \cdots$};

\node (B1) [below=1cm of A]{$\cdots\ast \Alg(D_{r-1},\A_{m_{r-1}})\big) \ast \big( \Alg(D_{r_1}, \A_{r_1}) \aprod_{r,r} \Alg(D_{r_2}, \A_{r_2})\big) \ast \big(\Alg(D_{r+1}, \A_{m_{r+1}}) \ast \cdots$};
\node (B) [below=1cm of A1] {};
\node (T) [right=3cm of B] {$\NAlg(S,A')$};

\draw[->] (B1) to node {$\pi$} (S);
\draw[->] (A1) to node {$\rho$} (T);

\draw[transform canvas={xshift=-0.5ex},->] (A1) to node[pos=.25,left] {$\tilde{\b}$} (B1);
\draw[transform canvas={xshift=0.5ex},->] (B1) to node[pos=.25,right] {$\tilde{\a}$} (A1);

\draw[transform canvas={yshift=+0.5ex},dashed,->] (S) to node[pos=.75,above] {$\a$} (T);
\draw[transform canvas={yshift=-0.5ex},dashed,->] (T) to node[pos=.75,below] {$\b$} (S);

\draw[->] (B1) to node [swap] {$\a'$} (T);
\draw[->] (A1) to node [swap] {$\b'$} (S);
\end{tikzpicture}
\end{equation*}
In this diagram, there are maps $\tilde{\a}$ and $\tilde{\b}$ which extend the assignments made by the gluing theorem in Eqn. \eqref{bdefeqn} by identity homomorphisms on the other components in each of the two free products. By construction, $\tilde{\a}\tilde{\b} = 1$ and $\tilde{\b}\tilde{\a} = 1$. The maps $\pi$ and $\rho$ are quotient maps, set $\a' := \rho\tilde{\a}$ and $\b' := \pi \tilde{\b}$. 

The maps $\a'$ and $\b'$ lift along $\pi$ and $\rho$ to maps $\a$ an $\b$ because they respect the (G1) and (G3) relations which determine the quotients. In more detail, they respect the (G1) relations because these relations state that two arcs in two disks are set to be equal when they are glued and the collection of such arcs is identical for the arc systems $A'$ and $A\backslash \{r\}$. The maps $\a'$ and $\b'$ respect the (G3) relations for the same reason; they are parameterized in the same way by the same arcs. It follows that there are unique maps $\a$ and $\b$ such that $\a\pi=\a'$ and $\b\rho = \b'$

Finally, these lifts are mutually inverse isomorphisms because $\a\pi = \a'$
implies that $\b\a\pi = \b\a'$. So $\b\a$ is the unique lift of $\b\a'$
along $\pi$. Now using the equations above,
$(\b\a)\pi = \b\a' = \b\rho \tilde{\a} = \b'\tilde{\a} = \pi \tilde{\b}\tilde{\a} = \pi
1 = 1\pi$,
so that $\b\a\pi = \b\a' = 1\pi$. In particular, $(\b\a)\pi = 1\pi$, so
uniqueness of lifts implies that $\b\a = 1$. By symmetry, $\a\b = 1$.
\end{proof}

The proposition below shows that the naive conjecture is functorial with
respect to the embedding theorem \ref{embthm}.

\begin{prop}\label{nembthm}
  Suppose that $i : (S', M', A') \hookrightarrow (S, M, A)$ is an inclusion of marked surfaces, each with enough marked intervals and $\pi_0(i\vert_{M'})$ is injective. Then there is an algebra homomorphism $i_*$ making the diagram below commute.
\begin{center}
\begin{tikzpicture}[scale=10, node distance=2.5cm]
\node (A) {$\NAlg(S',A')$};
\node (B) [right of=A] {$\NAlg(S,A)$};
\draw[->] (A) to node {$i_*$} (B);
\node (Ab) [below of=A] {$\Alg(S',A')$};
\node (Bb) [below of=B] {$\Alg(S,A)$};
\draw[right hook->] (Ab) to node {$j_*$} (Bb);
\draw[->>] (A) to node [swap] {$\ga_{A'}$} (Ab);
\draw[->>] (B) to node {$\ga_{A}$} (Bb);
\end{tikzpicture} 
\end{center}
Here $j_*$ is the embedding from Thm. \ref{embthm} and the maps $\ga_A$ and $\ga_{A'}$ are comparison maps. 
  \end{prop}
\begin{proof}
	  The map $i_* : \NAlg(S', A')\rightarrow \NAlg(S,A)$ is obtained by writing $\NAlg(S,A)$ as a quotient $r$ 
  $$\NAlg(S',A') \xto{p} \NAlg(S',A') \ast \Alg(D^2_1, \A_{m_1})\ast \cdots\ast \Alg(D^2_n, \A_{m_n}) \xto{r} \NAlg(S,A)$$
  where $\{(D^2_i, \A_{m_i})\}_{i=1}^n$ are the disks in $S\backslash A$ which are not contained in $S'\backslash A'$. So $i_* := rp$. 
The diagram commutes because all of the maps in the commutative diagram act
by identity on generating arcs. 
\end{proof}

\section{The Hall algebras of naive annuli}\label{nannulisec}

In this section the naive conjecture \ref{naiveconj} is proven for the
annuli $K_{m,n}$ with $m,n \geq 2$ marked intervals on each boundary
component. Section \ref{arcansec} introduces arc systems $\bar{\A}_{m,n}$
and $\A_{m,n}$ for each annulus $K_{m,n}$. The arc system $\A_{m,n}$ is
fully formal in the sense of Def. \ref{fullyformaldef} and the arc system
$\bar{\A}_{m,n}$ is full and has enough marked intervals in the sense of
Def. \ref{def:enough}. In \S\ref{naiverelsec} the presentation for the
algebra $\NAlg(K_{m,n}, \bar{\A}_{m,n})$ predicted by the naive conjecture
is written out in complete detail. The next two sections constitute the
technical heart of the paper. Section \ref{nanparamsec} constructs an
isomorphism between the composition subalgebra $\Alg(K_{m,n}, \A_{m,n})$
and the composition subalgebra associated to
a quiver $V_{m,n}$. Section \ref{naiveproofsec} uses this isomorphism, in
conjunction with Prop. \ref{nembthm} to show that the comparison map
$\ga_{\bA_{m,n}}$ is an isomorphism, proving the naive conjecture.

\subsection{Arc systems for annuli}\label{arcansec}

An annulus $S^1\times [0,1]$ with $m$ boundary arcs $\{E_1,\ldots, E_m\}$
on one boundary component and $n$ boundary arcs $\{F_1,\ldots,F_n\}$ on the other is denoted by $K_{m,n}$.
$$K_{m,n} := (S^1\times [0,1],M) \conj{ where } M := (S^1\times \{0, 1\}) \backslash (\{E_1,\ldots, E_m\} \sqcup \{F_1,\ldots, F_n\})$$
We require $m,n \geq 1$ so that  $K_{m,n}$ is finitary, see \cite[Prop. 4.10]{sillyman1}. The labels on the arcs $\{E_i\}_{i\in \ZZ/m}$ and
$\{F_i\}_{i\in \ZZ/n}$ are cyclically ordered according to the orientation
of the surface; so that $E_{i+1}$ follows $E_i$ and $F_{j+1}$ follows $F_j$ in the cyclic order induced by the orientation along the boundary. (This implies $\Hom(E_i, E_{i+1})=k$.)

There is a {\em standard arc decomposition} $\A_{m,n} :=\{S,T\}\sqcup \{P_1,\ldots,P_{m-1}\}\sqcup \{Q_1,\ldots, Q_{n-1}\}$  consisting of arcs which cut the annulus $K_{m,n}$ into disks.  
This arc system can be completed to a full arc system
\begin{equation}\label{fstdarceq}
  \bA := \A \sqcup \{E_1,\ldots,E_m\} \sqcup \{F_1,\ldots,F_n\}
  \end{equation}
by adding all of the missing boundary arcs to the standard arc system. 
For notational reasons, it will be convenient to set $P_m := T$, $Q_n := T$, $Q_0 := S$ and $Q_n := S$. See the figure below.

\begin{figure}[h]
\begin{overpic}[scale=1.4]
{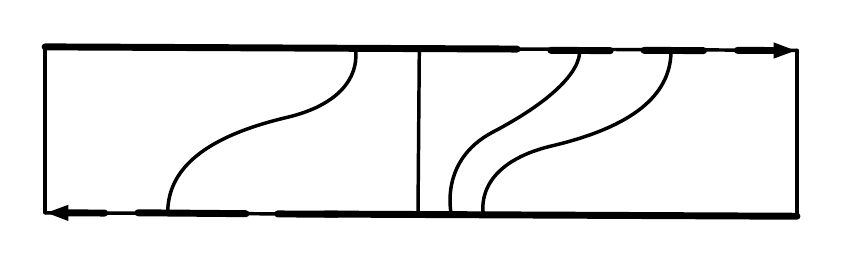}
\put(164,91){$S$}
\put(164,6){$S$}

\put(2,49){$T$}
\put(327,49){$T$}

\put(175,6){$P_1$}
\put(190,6){$P_2$}

 \put(42,6){$F_2$}
 \put(98,6){$F_1$}

\put(136,91){$Q_1$}

\put(210,91){$E_1$}
\put(248,91){$E_2$}
\put(285,91){$E_3$}
\end{overpic}
\caption{The annulus $K_{3,2}$ and the arc system $\bA_{3,2}$.}
\label{fig:standardarcs}
\end{figure}
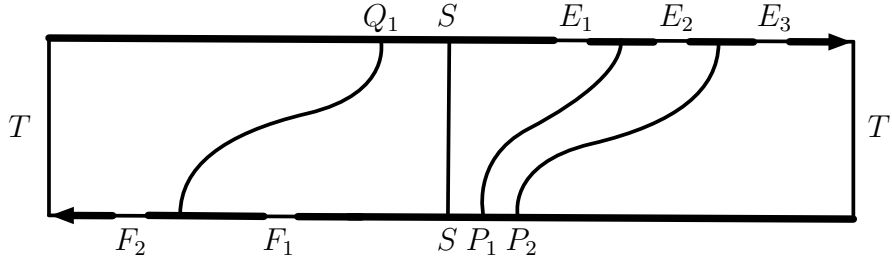

\begin{notation}
Many of the objects in this section are parameterized by fixed subscripts
$m,n$.  These subscripts will be written often enough to avoid confusion,
but not so often as to avoid nuisance.
  \end{notation}

\begin{defn}\label{disklabelsdef}
Each triple of arcs on the righthand side below uniquely determines the disk in $K_{m,n}\backslash \bA$ 
which will be labelled by the lefthand side.
\begin{align*}
C_j &:= (D^2,\{ P_j, P_{j+1}, E_{j+1}\}) \conj{ for } 0 \leq j < m  \notag \\
D_i &:= (D^2,\{ Q_i, Q_{i+1}, F_{i+1}\}) \conj{ for } 0 \leq i < n \notag
\end{align*}
\end{defn}

\begin{defn}\label{stdfoldef}
If $\Ga_{m,n}$ is the ribbon graph dual to $K_{m,n}$ then the {\em standard
  foliation data} $\l : \Ga \to \ZZ$ is determined by setting $\l(h) =1$ for
each half-edge dual to a boundary edge and $\l(h) = 0$
otherwise. This foliation data is balanced and so it is constant along the core curve of the annulus.
\end{defn}

\subsection{Naive relations for the composition subalgebra $\Alg(K, \bA)$}\label{naiverelsec}
In this section presentations for the algebras $\NAlg(K_{m,n}, \bA)$ introduced by Def. \ref{naivedef} are fully articulated.

\begin{prop}\label{naiveannulusrels}
  The naive algebra $\NAlg(K_{m,n},\bA)$ of the annulus $K_{m,n}$ is the $\QQ(q)$-algebra generated by suspensions of arcs 
$\ZZ\bA$ in the full standard arc system 
$$\bA = \{S,T\}\sqcup \{P_1,\ldots,P_{m-1}\}\sqcup \{Q_1,\ldots, Q_{n-1}\} \sqcup \{E_1,\ldots,E_m\} \sqcup \{F_1,\ldots,F_n\},$$
$$\normaltext{ and } Q_0 := S,\quad Q_n := T,\quad P_0 := S,\quad P_m := T$$
see Eqn. \eqref{fstdarceq}. 
These generators are subject to the relations listed below.
\begin{enumerate}
\item The three boundary arcs on each of the disks, $C_j$ and $D_i$ from Def.  \ref{disklabelsdef},
in the decomposition $K_{m,n}\backslash \bA$ satisfy relations (R1) and (R2) from Cor. \ref{diskprescor}.

\item Commutativity relations hold between distant arcs. In particular, for each $\ell\in\ZZ$, there are
\begin{enumerate}
\item $S$ and $T$ arc relations
\begin{align}
[S, E_{k,\ell}]_1 = 0  \nt{ for } 1 < k < m & \conj{ and }  [T, E_{k,\ell}]_1 = 0  \nt{ for } 1 < k < m \label{st1}\\
 [S, F_{k,\ell}]_1 = 0  \nt{ for } 1 < k < n & \conj{ and }  [T, F_{k,\ell}]_1 = 0  \nt{ for } 1 < k < n \notag
\end{align}
\item  $P$ and $Q$ arc relations 
\begin{align*}
[P_i, E_{k,\ell}]_1 &= 0  \nt{ for } 1 \leq i < m \nt{ and } k\ne i, i+1\\
[P_i, F_{k,\ell}]_1 &= 0  \nt{ for } 1 \leq i < m \nt{ and } k\ne 1, n \\
[Q_i, F_{k,\ell}]_1 &= 0  \nt{ for } 1 \leq i < n \nt{ and } k\ne i, i+1 \\
[Q_i, E_{k,\ell}]_1 &= 0  \nt{ for } 1 \leq i < n \nt{ and } k\ne 1, m \\
[P_i, Q_{k,\ell}]_1 &= 0  \nt{ for } 1 \leq i < m \nt{ and } 1 \leq k < n 
\end{align*}
\item Boundary arc relations
\begin{align}
  [E_i, E_{j,\ell}]_1 &= 0  \nt{ for } 1\leq i \leq m \nt{ and } j\ne i,i+1 \in \ZZ/m \label{ef1}\\
  [F_i, F_{j,\ell}]_1 &= 0 \nt{ for } 1 \leq i \leq n \nt{ and } j\ne i,i+1 \in \ZZ/n \notag\\
[E_i, F_{j,\ell}]_1 &= 0  \nt{ for } 1 \leq i \leq m \nt{ and } 1 \leq j \leq n \label{ef3}
\end{align}
where $E_{k,\ell} := \s^\ell E_k$, $F_{k,\ell} := \s^\ell F_k$ and relations are taken to be closed under the action of suspension, see Def. \ref{naivedef}
\end{enumerate}
\end{enumerate}
\end{prop}

The proposition is proven by a careful study of Fig. \ref{fig:standardarcs}
and knowledge of Def. \ref{naivedef}.

\subsection{The parameterization $\Phi$}\label{nanparamsec}
The key result in this section is Thm. \ref{naivethm} which shows that the
generator $G$ associated to the arc system $\A_{m,n}$ determines an
isomorphism between the composition subalgebra 
$\Alg(K_{m,n}, \bA_{m,n})$ and the composition subalgebra
$\DC(V_{m,n}^{\op})$ of the quiver $V^{op}_{m,n}$ introduced below. For the
most part, this is accomplished by showing that the isomorphism between
derived Hall algebras induced by the equivalence between $D^\pi\F(K, \A)$
and $D^b(\rmodule\End(G))$ commutes with inclusions of associated
composition subalgebras. The important stuff can be found in
Lem. \ref{phimaplemma}.  The rest follows from the identifications $\End(G)
\cong \rmodule kV$ and $\Alg(K, \A) \cong \Alg(K, \bA)$ which can be found
in Eqn. \eqref{rhoeqn} and the proof of Thm. \ref{naivethm} respectively.

The complement $K_{m,n}\backslash \A_{m,n}$ of the standard arc system is a disjoint union of disks $C_j$ and $D_i$ from Def. \ref{disklabelsdef}. The object $G_{m,n} := \oplus_{L\in \A_{m,n}} L$ generates the category $\DEF(K,\A)$ (see \cite[\S 3.4]{HKK}). 
Since $\bA$ is a full arc system and $\A$ is fully formal, the inclusion $\A\subset\bA$ induces an equivalence
\begin{equation}\label{vpeqn}
  \varphi : \DEF(K_{m,n}, \A) \xto{\sim} \DEF(K_{m,n}, \bar{\A}).
  \end{equation}
Since each of the disks in the collection $K_{m,n}\backslash \A_{m,n}$ contain exactly one boundary arc, an $E_i$ or an $F_j$, which is not in the arc system $\A$, the endomorphism algebra
\begin{equation}\label{kvmneqn}
\kk V_{m,n} := \End^*_{\DEF(K,\A)}(G_{m,n})
\end{equation}
is formal; this means that the higher $A_\infty$-operations vanish, $\mu_n=0$ for $n\geq 3$, see \cite[\S 3.4]{HKK}. By Def. \ref{fukcatdef}, $\kk V_{m,n}$ is the path algebra \cite[Def. 4.5]{Schiffler} of the quiver $V_{m,n}$ pictured below.

\begin{equation*}\begin{tikzpicture}[scale=10, node distance=2.5cm]
\node (S) {$P_0 = S = Q_0$};
\node (A) [right=1cm of S] {};

\node (A1) [above=1cm of A]{$P_1$};
\node (X) [left=3cm of A1]{$V_{m,n}:$};
\node (A2) [right=1cm of A1]{$P_2$};
\node (A3) [right=1cm of A2]{$\cdots$};
\node (A4) [right=1cm of A3]{$P_{m-1}$};

\node (B1) [below=1cm of A]{$Q_1$};
\node (B2) [right=1cm of B1]{$Q_2$};
\node (B3) [right=1cm of B2]{$\cdots$};
\node (B4) [right=1cm of B3]{$Q_{n-1}$};

\node (B) [below=1cm of A4] {};
\node (T) [right=1cm of B] {$P_m = T = Q_n$};

\draw[->] (S) to node {} (A1);
\draw[->] (S) to node {} (B1);

\draw[->] (A1) to node {} (A2);
\draw[->] (A2) to node {} (A3);
\draw[->] (A3) to node {} (A4);

\draw[->] (B1) to node {} (B2);
\draw[->] (B2) to node {} (B3);
\draw[->] (B3) to node {} (B4);

\draw[->] (B4) to node {} (T);
\draw[->] (A4) to node {} (T);
\end{tikzpicture}
\end{equation*}

Since $G_{m,n}$ generates the category $\DEF(K,\A)$, there is an equivalence of triangulated categories 
\begin{equation}\label{phieqn}
  \phi : \DEF(K, \A) \xto{\sim} D^b(\rmodule\kk V^{op})\conj{ given by } L \mapsto Hom^*(G,L)
  \end{equation}
from the split-closed Fukaya category Def. \ref{scfukcatdef} to the bounded derived category of finitely generated right $\kk V_{m,n}^{op}$-modules \cite[\S 7.6]{LH}.
Note that such a functor induces an isomorphism $\phi_*$ of associated derived Hall algebras, see \cite{XX, Toen} or \cite[Thm. 2.6]{sillyman1}.

The next proposition records what this functor does to arcs in the standard arc system.

\begin{prop}\label{peter1prop}
The functor  $\phi$ associates to each arc $L\in \A$ the right $kV^{op}$-module $\phi(L) = 1_L\kk V^{op}$ consisting of oriented paths in the quiver $V^{op}$ which begin at the vertex $L$.
\end{prop}
\begin{proof}
For each arc $L\in \A$, the value of $\phi$ can be computed since
\begin{align*}
\phi(L) &= Hom_{\DEF(K,\A)}(G,L)\\
&= Hom_{\DEF(K,\A)}(\opp_{L'\in \A} L' , L)\\
&= \opp_{L'\in \A} Hom_{\DEF(K,\A)}(L' , L)\\
&= \opp_{L'\in \A} Hom_{\F(K,\A)}(L' , L).
\end{align*}
So $\phi(L)$ has a basis given by the paths $\b : L' \to L$ in $V$, or equivalently, paths $\b : L \to L'$ in $V^{op}$. 
The module structure is determined by composition, a path $\a : X\to Y$ in $V^{op}$ acts on a path $\b$ in $V^{op}$ by
$$\b\cdot\a = \left\{\begin{array}{ll}
\b\a & \normaltext{ if } Y = L' \\ 
0 & \normaltext{ if } Y \ne L'. \\
\end{array}\right.$$
So as a right module over the path algebra $\kk V^{op}$, $\Phi(L)$ is the right projective module $1_L \kk V^{op}$ where $1_L$ is the idempotent associated to the vertex $L$. 
\end{proof}

There is an equivalence of abelian categories $\rho : \rmodule\kk V^{op} \xto{\sim} Rep_\kk(V^{op})$,
see \cite[Thm. 5.4]{Schiffler}. This map induces an equivalence 
\begin{equation}\label{rhoeqn}
  \rho_* : D^b(\rmodule\kk V^{op}) \xto{\sim} D^b(Rep_\kk(V^{op}))
  \end{equation}
  of associated derived categories.  In the proposition above, $\rho_*$
  takes the right module $1_L \kk V$ to a functor $V^{op} \to \Vect_{\kk}$
  from the opposite of the quiver $V$ to the category of vector spaces. The
  value of this functor at a vertex $K$ in the quiver $V$ is
  $Hom_{V^{op}}(L,K)$. For notational reasons, this paper will not distinguish between $1_LkV^{op}$ and $\rho_*(1_L kV^{op})$.

Composing the functors in  Eqns. \eqref{phieqn} and  \eqref{rhoeqn} produces the main object of our study, $\Phi := \rho_*\circ \phi$, which is an equivalence of categories
\begin{equation}\label{parameqn}
\Phi : \DEF(K_{m,n}) \xto{\sim} D^b(\Rep_\kk(V^{op}_{m,n})).
\end{equation}

The example below is included to illustrate which quiver modules are associated to arcs in $\A$ by the functor $\Phi$.

\begin{example}\label{k22ex}
Suppose that the surface is $K_{2,2}$. Then the equivalence
$\Phi : \DEF(K_{2,2}, \A) \xto{\sim} D^b(Rep_{\kk}(V_{2,2}^{op}))$
 associates to each arc $L\in \A=\{S,T,Q_1,P_1\}$, a functor $\Phi(L) : V^{op} \to Vect_{\kk}$ in $Rep_{\kk}(V^{op})$. By Prop. \ref{peter1prop} above, the value of $\Phi(L)$ at a vertex $K$ is the $\kk$-vector space $Hom_{V^{op}}(L,K)$. This is the set of $\kk$-linear combinations of oriented paths in the quiver $V^{op}_{2,2}$ which begin at the vertex $L$ and end at the vertex $K$. 
For $L\in \A$, these modules are pictured below.
$$\begin{array}{cccc}
\begin{tikzpicture}[scale=10, node distance=2.5cm]
\node (S) {$\kk$};
\node (A) [right=.5cm of S] {};
\node (A1) [above=.5cm of A]{$0$};
\node (X) [left=.5cm of A1]{$\Phi(S):$};
\node (B1) [below=.5cm of A]{$0$};
\node (T) [right=.5cm of A] {$0$};
\draw[<-] (S) to node {} (A1);
\draw[<-] (S) to node {} (B1);
\draw[<-] (A1) to node {} (T);
\draw[<-] (B1) to node {} (T);
\end{tikzpicture} &
\begin{tikzpicture}[scale=10, node distance=2.5cm]
\node (S) {$\kk$};
\node (A) [right=.5cm of S] {};
\node (A1) [above=.5cm of A]{$\kk$};
\node (X) [left=.5cm of A1]{$\Phi(P_1):$};
\node (B1) [below=.5cm of A]{$0$};
\node (T) [right=.5cm of A] {$0$};
\draw[<-] (S) to node {} (A1);
\draw[<-] (S) to node {} (B1);
\draw[<-] (A1) to node {} (T);
\draw[<-] (B1) to node {} (T);
\end{tikzpicture} &
\begin{tikzpicture}[scale=10, node distance=2.5cm]
\node (S) {$\kk$};
\node (A) [right=.5cm of S] {};
\node (A1) [above=.5cm of A]{$0$};
\node (X) [left=.5cm of A1]{$\Phi(Q_1):$};
\node (B1) [below=.5cm of A]{$\kk$};
\node (T) [right=.5cm of A] {$0$};
\draw[<-] (S) to node {} (A1);
\draw[<-] (S) to node {} (B1);
\draw[<-] (A1) to node {} (T);
\draw[<-] (B1) to node {} (T);
\end{tikzpicture} &
\begin{tikzpicture}[scale=10, node distance=2.5cm]
\node (S) {$\kk^2$};
\node (A) [right=.5cm of S] {};
\node (A1) [above=.5cm of A]{$\kk$};
\node (X) [left=.5cm of A1]{$\Phi(T):$};
\node (B1) [below=.5cm of A]{$\kk$};
\node (T) [right=.5cm of A] {$\kk$};
\draw[<-] (S) to node {} (A1);
\draw[<-] (S) to node {} (B1);
\draw[<-] (A1) to node {} (T);
\draw[<-] (B1) to node {} (T);
\end{tikzpicture}
\end{array}$$
\end{example}

The lemma below is standard triangulated category stuff (for a proof see \cite[Prop. 1.1.20]{Neeman}).

\begin{lemma}\label{trilem}
Suppose the diagram below is a commutative square 
\begin{center}
\begin{tikzpicture}[scale=10, node distance=1.5cm]
\node (A) {$A$};
\node (B) [right of=A] {$B$};
\draw[->] (A) to node {$\a$} (B);
\node (Ab) [below of=A] {$C$};
\node (Bb) [below of=B] {$D$};
\draw[->] (Ab) to node {$\b$} (Bb);
\draw[->] (A) to node [swap] {$\psi$} (Ab);
\draw[->] (B) to node {$\varphi$} (Bb);
\end{tikzpicture} 
\end{center}
in a triangulated category $\aT$. If the vertical maps $\psi$ and $\varphi$ are isomorphisms then there is an isomorphism $C(\a) \cong C(\b)$ between the cones of the horizontal maps.
  \end{lemma}

\begin{defn}\label{def:newarcN}
The {\em special arc $N$} is pictured as a dashed line below. This arc is contained in the boundary of two embedded disks $D$ and $D'$. 
\begin{center}\label{fig:newarcN}
\begin{overpic}[scale=1.4]
{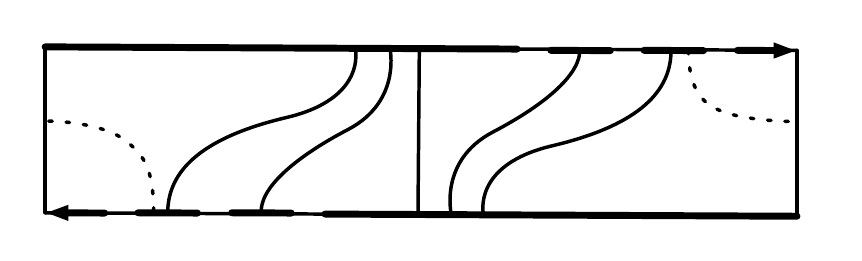}
\put(325,53){$N$}
\put(4,53){$N$}
\put(65,65){$D$}
\put(265,30){$D'$}
\put(284,91.5){$E_m$}
\put(135,91.5){$Q_{n-1}$}
\put(42.5,6){$F_n$}
\put(188,6){$P_{m-1}$}
\end{overpic}
\end{center}
\end{defn}

Functoriality implies that the equivalence $\Phi$ from Eqn. \eqref{parameqn} induces an isomorphism 
$$\Phi_* : \tDHa(\DEF(K_{m,n})) \to \tDHa (D^b(\Rep_\kk(V^{op}_{m,n}))).$$
Thm. \ref{naivethm} will use Lem. \ref{phimaplemma} below to show that $\Phi_*$  induces an isomorphism between composition subalgebras. 

The lemma below computes the value of $\Phi_*$ on generating arcs $L\in \A$. Recall that $z_L$ is the 1-dimensional simple module associated to the vertex $L\in V^{op}$.

\begin{lemma}\label{phimaplemma}
\begin{enumerate}
\item The isomorphism $\Phi_*$ maps the following arcs onto the 1-dimensional simple modules of the quiver $V^{op}_{m,n}$
\begin{align}
\Phi_*(S) &= z_S  & \notag\\
\Phi_*(E_i) &= z_{P_i} & \normaltext{ for } 1 \leq i < m \label{pieq}\\
\Phi_*(F_j) &= z_{Q_j} & \normaltext{ for } 1 \leq j < n \label{qjeq}\\
\Phi_*(N) &= z_T & \label{tsimpeq}
\end{align}
\item The value of $\Phi_*$ on any other arc $X\in \bA$ has an expression in terms of simple modules
\begin{align}
\Phi_*(P_i) &= [\Phi_*(P_{i-1}), z_{P_i}]_q & \normaltext{ for } 0 < i < m \label{phi1eq}\\
\Phi_*(Q_j) &= [\Phi_*(Q_{j-1}), z_{Q_j}]_q & \normaltext{ for } 0 < j < n \notag\\
\Phi_*(E_m) &= [\Phi_*(Q_{n-1}), z_T]_q & \label{phi2eq}\\
\Phi_*(F_n) &= [\Phi_*(P_{m-1}), z_T]_q & \notag\\
\Phi_*(T) &= [\Phi_*(P_{m-1}),[\Phi_*(Q_{n-1}), z_T]_q ]_q & \label{phi3eq} \\
        &= [\Phi_*(Q_{n-1}),[\Phi_*(P_{n-1}), z_T]_q ]_q. &\notag 
\end{align}
\end{enumerate}
  \end{lemma}
\begin{proof}
The proof has four parts. The first three parts establish the equations in the list (1)
above.  The first part concerns the simple module $z_S$ at the vertex
$S$. The second part shows that the simple modules at vertices $P_i$ and
$Q_j$ correspond to the arcs $E_i$ and $F_j$ respectively. Since the
argument in the second part is repeated several times, more details are
given there.  The third case identifies the simple module at the vertex $T$
as the image of the arc $N$ from Def. \ref{def:newarcN}.  The last part uses the distinguished triangles
from earlier computations to establish the equations in list (2) above.

{\it Part 1.} The arc $S$ corresponds to the simple module $z_S$ because the projective $\Phi(S) = 1_SkV^{op}$ is 1-dimensional, see Ex. \ref{k22ex} above.

{\it Part 2.}  In the category $Rep_{\kk}(V^{op})$, there is a short exact sequence
\begin{equation}\label{seseqn}
0 \to 1_{P_{i-1}}kV^{op}\to 1_{P_i}kV^{op} \to z_{P_i} \to 0
\end{equation}
for $1\leq i < m$, see \cite[Rmk. 1.3]{Schiffler}. These short exact sequences determine distinguished triangles of the form
\begin{equation}\label{disreqpeqn}
\cdots \to 1_{P_{i-1}}kV^{op}\to 1_{P_i}kV^{op} \to z_{P_i} \to \cdots
\end{equation}
in the derived category $D^b(Rep_{\kk}(V^{op}))$. By Prop. \ref{peter1prop} above, there are isomorphisms $1_{P_i}kV^{op} \xto{\sim} \Phi(P_i)$. Under this identification, the map $1_{P_{i-1}}kV^{op}\to 1_{P_i}kV^{op}$ in Eqn. \eqref{disreqpeqn} corresponds to the map $\Phi(p_{i-1})$ where $p_{i-1} : P_{i-1} \to P_i$ is the boundary arc from $P_{i-1}$ to $P_i$ in the annulus $K_{m,n}$. So the diagram below commutes.
\begin{center}
\begin{tikzpicture}[scale=10, node distance=2cm]
\node (B) {$1_{P_{i-1}} kV^{op}$};
\node (C) [right=2cm of B] {$1_{P_i}kV^{op}$};
\node (D) [right=2cm of C] {$z_{P_i}$};
\draw[->] (B) to node {} (C);
\draw[->] (C) to node {} (D);

\node (Bb) [below of=B] {$\Phi(P_{i-1})$};
\node (Cb) [below of=C] {$\Phi(P_i)$};
\node (Db) [below of=D] {$C(\Phi(p_{i-1}))$};
\draw[->] (Bb) to node {} (Cb);
\draw[->] (Cb) to node {} (Db);

\draw[->] (B) to node {} (Bb);
\draw[->] (C) to node {} (Cb);
\draw[->, dashed] (D) to node {$a$} (Db);
\end{tikzpicture} 
\end{center}
By Lem. \ref{trilem}, there is an isomorphism $a : z_{P_i} \xto{\sim} C(\Phi(p_{i-1}))$. Since $\Phi$ is a map of triangulated categories, there is an isomorphism $b : C(\Phi(p_{i-1})) \xto{\sim} \Phi(C(p_{i-1}))$.

On the other hand, the disk $C_{i-1} \subset K$ determined by the arcs
$\{P_{i-1}, P_i, E_i\}$ 
satisfies the embedding criteria of Thm. \ref{embthm},
so the associated functor $D^\pi\F(C_{i-1}) \hookrightarrow D^\pi\F(K_{m,n})$ is full and faithful. Since it maps the distinguished triangle
$$  \cdots \to E_{i,1} \to P_{i-1,0} \xto{p_{i-1}} P_{i,0} \to E_{i,0} \to P_{i-1,-1}\to \cdots$$
to the distinguished triangle associated to the boundary path $p_{i-1}$, Lem. \ref{trilem} shows that there is an isomorphism $c : C(p_{i-1}) \xto{\sim} E_i$ in $D^\pi\F(K,\bA)$. Combining these isomorphisms shows that
$$z_{P_i} \cong_a C(\Phi(p_{i-1})) \cong_b \Phi(C(p_{i-1})) \cong_{\Phi(c)} \Phi(E_i).$$
So $\Phi_*(E_i) = z_{P_i}$ for $1\leq i < m$.  The proof of Eqn. \eqref{qjeq} follows from the same argument because the computation is symmetric for the arcs $P_i$ and $Q_j$.

{\it Part 3.} The proof that $\Phi_*(N) = z_T$ requires two short exact sequences.
Each short exact sequence gives a relation by an argument analogous to the one in Part 2. The two relations combine at the end.

First there is a short exact sequence
\begin{equation}\label{tauseq1}
0 \to 1_{P_{m-1}}kV^{op} \to 1_TkV^{op} \to X \to 0
\end{equation}
where $X := 1_TkV^{op}/1_{P_{m-1}}kV^{op}$ is the module $V^{op} \to \Vect_{\kk}$ in the category $Rep_{\kk}(V^{op})$ described by the diagram below.
\begin{center}
\begin{tikzpicture}[scale=10, node distance=2.5cm]
\node (S) {$\kk$};
\node (A) [right=.5cm of S] {};
\node (A1) [above=.5cm of A]{$0$};
\node(A2) [right=.5cm of A1]{$0$};
\node(A3) [right=.5cm of A2]{$\cdots$};
\node(A4) [right=.5cm of A3]{$0$};
\draw[->] (A2) to node {} (A1);
\draw[->] (A3) to node {} (A2);
\draw[->] (A4) to node {} (A3);

\node (X) [left=2cm of A1]{$X:$};
\node (B1) [below=.5cm of A]{$\kk$};

\node(B2) [right=.5cm of B1]{$\kk$};
\node(B3) [right=.5cm of B2]{$\cdots$};
\node(B4) [right=.5cm of B3]{$\kk$};
\draw[->] (B2) to node {} (B1);
\draw[->] (B3) to node {} (B2);
\draw[->] (B4) to node {} (B3);

\node (X) [above=.5cm of B4] {};
\node (T) [right=.5cm of X] {$\kk$};
\draw[<-] (S) to node {} (A1);
\draw[<-] (S) to node {} (B1);
\draw[<-] (A4) to node {} (T);
\draw[<-] (B4) to node {} (T);
\end{tikzpicture}
\end{center}
Since the inclusion $1_{P_{m-1}}kV^{op} \to 1_TkV^{op}$ agrees with $\Phi(p_{m-1})$
under the identifications in Prop. \ref{peter1prop}, the isomorphism
$\Phi(E_m) \cong X$ follows from Lem. \ref{trilem}.

Secondly, the boundary arc $\tau : Q_{n-1} \to E_m$ determines a short exact sequence
\begin{equation}\label{tauseq2}
0 \to 1_{Q_{n-1}}kV \xto{\Phi(\tau)}  X \to z_T \to 0
\end{equation}
in the category $\Rep_k(V^{op})$. So $C(\Phi(\tau)) \cong z_T$. However, since the disk $D$ in Def. \ref{def:newarcN} determined by the arcs $\{N, E_m, Q_{n-1}\}$ satisfies the embedding criteria of Thm. \ref{embthm}, Lem. \ref{trilem} implies that $C(\tau) \cong N\!$. Combining isomorphisms completes proof of Eqn. \eqref{tsimpeq},
$$\Phi(N) \cong \Phi(C(\tau))\cong C(\Phi(\tau)) \cong z_{T}.$$

{\it Part 4.} This section contains the proofs of Eqns. \eqref{phi1eq}, \eqref{phi2eq} and \eqref{phi3eq}. All of the other equations in the list (2) are computed in an analogous way by symmetry.

For Eqn. \eqref{phi1eq}, the identification $\Phi(P_i) \cong 1_{P_i}V$
follows from Prop. \ref{peter1prop}. This allows us to write the short exact sequences
\eqref{seseqn} as $0\to \Phi(P_{i-1}) \to \Phi(P_i) \to z_{P_i} \to 0$ so that Rmk. \ref{keyrelrmk} gives the relation $\Phi_*(P_i) = [\Phi_*(P_{i-1}), z_{P_i}]_q$ for $0 < i < m$.

To see Eqn. \eqref{phi2eq}, combining the short exact sequence \eqref{tauseq2} and the isomorphism $\Phi(E_m) \cong X$ from Part 2 above gives the short exact sequence $0\to \Phi(Q_{n-1}) \to \Phi(E_m) \to z_{T} \to 0$ so Rmk. \ref{keyrelrmk} implies $\Phi_*(E_m) = [\Phi_*(Q_{n-1}), z_{T}]_q$.

For Eqn. \eqref{phi3eq}, the short exact sequence \eqref{tauseq1} and Rmk. \ref{keyrelrmk} give the relation $T = [P_{m-1}, E_m]_q$ from which we obtain
$\Phi_*(T) = [\Phi_*(P_{m-1}), \Phi_*(E_m)]_q$.
\end{proof}

Before applying the lemma above to Thm. \ref{naivethm} below, recall the definition of the composition subalgebra of a quiver $Q$.

\begin{defn}\label{compalgdef}
The {\em composition subalgebra} $\DC(Q)$ of the Hall algebra 
$$\DC(Q) \subset \tDHa(D^b(\Rep_\kk(Q)))$$
is the subalgebra generated by suspensions of the 1-dimensional simple modules $z_{i,n} := z_i[n]$ for each vertex $i$ and $n\in \ZZ$.
\end{defn}

The theorem below is the principal application of Lem. \ref{phimaplemma}.

\begin{thm}\label{naivethm}
  The derived equivalence $\Phi : D^\pi\F(K,\A) \xto{\sim} D^b(Rep(V^{op}))$ induces an isomorphism 
  $$\Phi_* : \Alg(K, \bA) \xto{\sim} \DC(V^{op})$$
of composition subalgebras of derived Hall algeras.
\end{thm}
\begin{proof}
The proof involves two separate tasks, each of which follow from arguments
of the same type. The functoriality of the derived Hall algebra is used to
construct an isomorphism and then shown to induce a map between composition
subalgebras by looking at special relations in either case.

The first step is to construct the isomorphism $\Alg(K, \A) \cong \Alg(K, \bA)$. The
 relation $\A\subset\bA$ determines a functor $\F(K, \A) \to \F(K,\bA)$ which gives a functor $\DEF(K, \A) \xto{\sim} \DEF(K, \bA)$. The latter is an equivalence, see Rmk. \ref{ffrmk}. Restricting the map induced by functoriality of the derived Hall construction, produces a monomorphism $\Alg(K, \A)\hookrightarrow \tDHa (D^\pi\F(K, \bA))$ which factors through $\Alg(K, \bA)$ by definition. The resulting map 
$\Alg(K, \A) \to \Alg(K, \bA)$  is necessarily an injective homomorphism, it is onto
because each boundary arc $E_{i}$ is contained the disk $C_{i-1}$ determined by
  the arcs $\{P_{i-1}, P_i, E_i\}$. Since $C_{i-1} \subset K$ satisfies the embedding criteria in Thm. \ref{embthm}, the relation $E_i = [P_i, \s^{-1}P_{i-1}]_q$ must hold in $\Alg(K, \bA)$. The same argument applies to the generators $F_j$.

Now functoriality of the derived Hall construction implies that $\Phi$
induces an isomorphism $\Phi_* : \tDHa (D^\pi\F(K,\A)) \xto{\sim} \tDHa (D^b(Rep(V^{op})))$.
As an isomorphism, $\Phi_*\vert_{\Alg(K,\bA)}$ remains injective when restricted 
to the composition subalgebra $\Alg(K,\bA) \subset \tDHa(D^\pi\F(K,\A))$. On the other hand, the list (1) in Lem. \ref{phimaplemma} implies that
$$\{ z_S, z_{P_i}, z_{Q_j}, z_T : 1 \leq i < n, 1\leq j < m \} \subset \im \Phi_*\vert_{\Alg(K,\bA)}$$
and since $\Phi_*$ commutes with suspensions, it follows that the image of $\Phi_*\vert_{\Alg(K,\bA)}$ contains the generators of the composition subalgebra $\DC(V^{op})\subset \im \Phi_*\vert_{\Alg(K,\bA)} \subset \tDHa (D^b(Rep(V^{op})))$. 

On the other hand, the list (2) in Lem. \ref{phimaplemma} shows that the image
of any other arc generating $\Alg(K,\bA)$ is contained in the subalgebra
$\DC(V^{op})$ generated by the simple modules $z_{i,n}$ for $i\in \bA$.  So $\Phi_*(\bA) \subset \DC(V^{op})$.

It follows that $\im \Phi_*\vert_{\Alg(K,\bA)} = \DC(V^{op})$,  so that an isomorphism between composition subalgebras is obtained by restriction, $\Phi_* := \Phi_*\vert_{\Alg(K,\bA)} : \Alg(K, \bA) \xto{\sim} \DC(V^{op})$.
\end{proof}

\begin{rmk}\label{iammyowngrandparmk}
Although an equivalence $\Phi_*$ exists for the non-naive annuli, $K_{m,n}$
with $n=1$ or $m=1$, the equivalence does not descend to a map between
composition subalgebras as above.
  \end{rmk}

\subsection{Relations for the composition subalgebra $\DC(V)$}\label{relcompsec}

Thm \ref{naivethm} shows that the composition subalgebra $\Alg(K,\bA)$ is isomorphic to the composition subalgebra $\DC(V^{op})$. The purpose of this section is to introduce relations for the latter algebra.

Prop. \ref{kalgdef} below is due to Hernandez-Leclerc \cite[Prop. 8.1]{HL},
it relies on To\"{e}n \cite[Prop. 7.1]{Toen} which in turn extends Ringel
\cite{Ringel}.  These relations agree with those in \cite[\S 2.2]{sillyman1}.

Recall that $z_L$ is the 1-dimensional simple module associated to the vertex $L\in V^{op}$.

\begin{prop}\label{kalgdef}
  The composition subalgebra $\tCHa(Q)$ of an acyclic quiver is generated by
  the symbols $z_{i,n}$ parameterized by vertices $i \in I$ and $n\in\ZZ$ subject to the   relations below.
\begin{description}
\item[(K0)\namedlabel{k0:itm}{\lab{K0}}] for $n\in\ZZ$, if $i\cdot j = 0$
  $$[z_{i,n}, z_{j,n+k}]_{1} = 0$$
\item[(K1)\namedlabel{k1:itm}{\lab{K1}}] for $n\in \ZZ$, if $i\cdot j = -1$
\begin{align*}
[z_{i,n}, [z_{i,n},z_{j,n}]_{q^{\mp 1}}]_{q^{\pm 1}} &= 0\\
[z_{i,n}, z_{j,n+k}]_{q^{(-1)^k}} &= 0 \conj{ for } k\geq 1
\end{align*}
\item[(K2)\namedlabel{k2:itm}{\lab{K2}}] for $n\in\ZZ$, if $i\cdot j = 2$
  $$[z_{i,n}, z_{j,n+k}]_{q^{(-1)^k i\cdot j}} = \d_{k,1}\d_{i,j} \frac{q^{-1}}{q^2-1}\conj{ for } k \geq 1$$
\end{description}
where the pairing $i \cdot j := \inp{z_{i,0}, z_{j,0}} + \inp{z_{j,0}, z_{i,0}}$ is as in \S\ref{hallsec}.
\end{prop}

The remainder of this section applies this proposition to the
quivers $V_{m,n}$ from \S \ref{nanparamsec}.  By the proposition,
the graph $V_{m,n}$ together with the Euler form on the Grothendieck
group of its representation category suffice to determine a presentation
for $\DC(V^{op})$.

The Grothendieck group $K_0(\Rep_\kk(V_{m,n}^{op}))$ is given by the lattice
spanned by the simple modules $p_i$ for $0\leq i \leq m$ and $q_j$ for
$0\leq j \leq n$ associated to the vertices of $V_{m,n}^{op}$ after imposing
the relations $p_0 = q_0$ and $p_m = q_n$.

Recall that the Euler form in Eqn. \eqref{eulerformeq} of a quiver such as
$V_{m,n}^{op}$ can be computed from the graph \cite[Prop. 8.4]{Schiffler}. 
If $m>1$ or $n> 1$ then the Euler form is determined by the equations
\begin{align*}
p_i\cdot p_i = 2 \normaltext{ for } 0\leq i \leq m & \conj{ and }  q_i \cdot q_i = 2  \normaltext{ for } 0 \leq i \leq n,\\
p_i \cdot p_{i+1} = -1 \normaltext{ for } 0 \leq i < m & \conj{ and }  q_i \cdot q_{i+1}  = -1 \normaltext{ for } 0 \leq i < n.
\end{align*}
All other vectors pair to zero. These assignments in conjunction with Prop. \ref{kalgdef} above give
presentations for the composition subalgebras $\DC(V^{op})$.

\subsection{Proof of the naive conjecture}\label{naiveproofsec}
In this section, Thm. \ref{naiveisothm} shows that the comparison map
$\ga_{\bA} : \NAlg(K_{m,n}, \bA) \to \Alg(K_{m,n}, \bA)$ is an isomorphism
when $m, n \geq 2$. This establishes the naive conjecture for annuli.

Before proving the theorem it is useful to check that certain relations hold
in the algebras of interest.
\begin{lemma}\label{ddplemma}
When $m,n\geq 2$, the following relations hold in both the composition subalgebra $\Alg(K_{m,n}, \bA)$ and the naive algebra $\NAlg(K_{m,n},\bA)$.
\begin{enumerate}
\item  Disk relations from disks $C_{m-1}$ and $D_{n-1}$ determined by the arcs $\{E_m, T, P_{m-1}\}$ and $\{F_n, T, Q_{n-1}\}$
\begin{align*}
  T &= [P_{m-1}, E_m]_q                &   T &= [Q_{n-1}, F_n]_q \\
  E_m &= [T, \s^{-1}P_{m-1}]_q          &   F_n &= [T, \s^{-1} Q_{n-1}]_q\\
  P_{m-1} &= [\s E_m, T]_q             &   Q_{n-1} &= [\s F_n , T]_q
\end{align*}
\item  Disk relations from disks $D$ and $D'$ from Def. \ref{def:newarcN} determined by the arcs 
$\{E_m, Q_{n-1}, N\}$ and $\{F_n, P_{m-1}, N\}$
\begin{align*}
  E_m &= [Q_{n-1}, N]_q          & F_n &= [P_{m-1}, N]_q\\
  N &= [E_m,\s^{-1} Q_{n-1}]_q              & N &= [F_n, \s^{-1} P_{m-1}]_q \\
  Q_{n-1} &= [\s N, E_m]_q                 & P_{m-1} &= [\s N, F_n]_q
  \end{align*}
\end{enumerate}
  \end{lemma}
\begin{proof}
All of the relations follow from the observation that they hold in the disk algebras of embedded disks together with Prop. \ref{nembthm}.

  \end{proof}

When $m=1$ or $n=1$ the disk containing the unique boundary arc in
$K_{m,n}\backslash \bA$ does not satisfy the embedding criteria of
Thm. \ref{embthm} so the annulus is not naive. The theorem below shows that
the naive conjecture \ref{naiveconj} holds when $m,n \geq 2$.

\begin{theorem}\label{naiveisothm}
When $m,n\geq 2$, the map $\ga_{\bA}$ is an isomorphism.
\end{theorem}

\begin{proof}
The map $\ga_{\bA} : \NAlg(K_{m,n}, \bA)\twoheadrightarrow \Alg(K_{m,n}, \bA)$ is onto by Thm. 5.33 \cite{sillyman1}. By Thm. \ref{naivethm}, there is an isomorphism $\Phi_* : \Alg(K_{m,n}, \bA) \xto{\sim} \DC(V^{op})$. So there is a diagram,
\begin{center}
\begin{tikzpicture}[scale=10, node distance=2cm]
\node (A) {$\NAlg(K_{m,n}, \bA)$};
\node (X) [below=2.2 of A] {};
\node (B) [left=1 of X] {$\Alg(K_{m,n}, \bA)$};
\node (C) [right=1 of X] {$\DC(V_{m,n}^{op})$};
\draw[->] (A) to node [swap] {$\ga_{\bA}$} (B);
\draw[->] (B) to node  {$\Phi_*$} (C);
\draw[->, bend right=15] (A) to node [swap] {$\psi$} (C); 
\draw[->, bend right=15, dashed] (C) to node [swap] {$\varphi$} (A); 
\end{tikzpicture} 
\end{center}
in which $\psi := \Phi_* \circ \ga_{\bA}$. In order to show that $\ga_{\bA}$ is injective, it suffices to construct a homomorphism $\varphi : \DC(V^{op}_{m,n}) \to \NAlg(K_{m,n}, \bA)$ which satisfies 
\begin{equation}\label{leftinveq}
  \varphi\circ \psi = 1_{\NAlg(K,\bA)}.
  \end{equation}
The proof consists of four steps. The first step is to define the maps $\psi$ and $\varphi$ on generators. Since $\psi = \Phi_*\circ \ga_{\bA}$ by definition, the map $\psi$ is a homomorphism. In Step \#2, $\varphi$ is shown to be a homomorphism. In Step \#3, Eqn. \eqref{leftinveq} is shown to hold.

{\it Step \#1.} Since $\ga_{\bA}\vert_{\bA} = 1$ and the commutative diagram is $\s$-equivariant,  the homomorphism $\psi = \Phi_*\circ \ga_{\bA}$ agrees with $\Phi_*$ on generating arcs $\bA$ and the assignments below follow from list (1) of Lem. \ref{phimaplemma}. 
\begin{align}
\psi(S) &= z_S  & \psi(N) &= z_T  \notag\\
\psi(E_i) &= z_{P_i} & \psi(F_j) &= z_{Q_j} \notag
\end{align}
for $1 \leq i < m$ and $1 \leq j < n$. The value of $\psi$ on other arcs is determined by the inductive formula in list (2) of Lem. \ref{phimaplemma}.
\begin{align}
\psi(P_i) &= [\psi(P_{i-1}), z_{P_i}]_q & \normaltext{ for } 0 < i < m \label{psi1eq}\\
\psi(Q_j) &= [\psi(Q_{j-1}), z_{Q_j}]_q & \normaltext{ for } 0 < j < n \notag\\
\psi(E_m) &= [\psi(Q_{n-1}), z_T]_q & \label{psi2eq}\\
\psi(F_n) &= [\psi(P_{m-1}), z_T]_q & \notag\\
\psi(T) &= [\psi(P_{m-1}),[\psi(Q_{n-1}), z_T]_q ]_q & \label{psi3eq} 
\end{align}

Since the map $\varphi$ must invert the choices for $\psi$ on generating arcs, the assignments below are 
prescribed.
\begin{align}
  \vp(z_S) &= S  & \vp(z_T) &= N \label{defofphi}\\
  \vp(z_{P_i}) &= E_i &  \vp(z_{Q_j}) &= F_j \notag
\end{align}
where $1\leq i < m$ and $1\leq j < n$. (Note since $N\not\in\bA$, technically the definition of $\vp$ uses Lem. \ref{ddplemma} to express $N$ in terms of generators in $\bA$.)

{\it Step \#2.} In order to prove that the definition of $\varphi$  in Step \#1 determines a homomorphism 
$$\varphi : \DC(V^{op}) \to \NAlg(K,\bA)$$ 
it suffices to show that the Eqns. \eqref{k0:itm}, \eqref{k1:itm} and
\eqref{k2:itm} from Prop. \ref{kalgdef} vanish in the image of $\vp$. To demonstrate that this is so, in each case, this proof either invokes a far commutativity \eqref{g3:itm}-relation from Def. \ref{naivedef} or uses Prop. \ref{nembthm} to show that the relation holds in the Hall algebra of an embedded disk.

{\it For Eqn. \eqref{k0:itm}:} If $z$ and $w$ are 1-dimensional simple modules and the Euler form is $z\cdot w = 0$ then we will show that the relation 
$$\vp[z, \s^kw]_{1} = 0 \conj{ for } k \in \ZZ$$
holds in $\NAlg(K,\bA)$. There are several cases,
\begin{enumerate}
\item $p_i \cdot q_j = 0$ for all $0<i< m$ and $0<j<n$ with $i\ne j$.  This case follows immediately from Eqn. \eqref{ef3} because  $\vp[z_{P_i}, z_{Q_j,\ell}]_1 = [E_i, F_{j,\ell}]_1 = 0.$

\item $p_i \cdot p_j = 0$ for all $0\leq i < m$ and, $j< i-1$ or $j > i+1$.
If $i = 0$ and $1 < j < m$ then the relation becomes $\vp[z_{P_0}, z_{P_j, \ell}]_1 = [S, E_{j, \ell}]_1$ which follows from Eqn. \eqref{st1}.  If $0< i <j < m$ then $\vp[z_{P_i}, z_{P_j, \ell}]_1 = [E_i, E_{j,\ell}]_1 = 0$ by Eqn. \eqref{ef1}. Similarly, the cases $0= j< i < m$ and $0< j< i < m$ are covered by Eqn. \eqref{st1} and Eqn. \eqref{ef1}. 

Two cases remain, in both $j=m$ and either $i=0$ or $0 < i < m-1$. In the image of $\vp$ they are
\begin{equation}\label{lastoneseqn}
  [S, \s^\ell N]_1 = 0 \conj{ and } [E_i, \s^\ell N]_1 = 0
  \end{equation}
respectively. Form a new arc system $\bA'$ by removing $T$ from $\bA$ and replacing it with $N$; $\bA' := (\bA\backslash \{T\}) \sqcup \{N\}$. This Pachner move preserves the disk embedding property because the disk $(D^2, \{E_m, Q_{n-1}, F_n, P_{m-1}\})$ is embedded fully faithfully in $K_{m,n}$, see Def. \ref{def:newarcN}. By Thm. \ref{naiveeqvthm} there is an isomorphism
$$f : \NAlg(K,\bA) \xto{\sim} \NAlg(K, \bA')$$
associated to this Pachner move. Since the arcs do not share marked intervals, the relations Eqn. \eqref{lastoneseqn} occur as \eqref{g3:itm} relations in $\NAlg(K, \bA')$, but $f$ is an isomorphism which satisfies $f(S) = S$, $f(E_i) = E_i$ and $f(N) = N$.

\item $q_i \cdot q_j = 0$ for all $0\leq i < n$ and, $j< i-1$ or $j > i+1$. The proof is symmetric to that of case (2) above.

\end{enumerate}

{\it For Eqn. \eqref{k1:itm}:} Suppose that $\vp(z)$ and $\vp(w)$ appear as
consecutive boundary arcs in a disk $R\hookrightarrow K$ that satisfies the embedding criteria in Prop. \ref{nembthm}. Then the relations
\begin{align}
\vp[z, [z,w]_{q^{\mp 1}}]_{q^{\pm 1}} &= 0 \label{k1a}\\
\vp[z, \s^k w]_{q^{(-1)^k}} &= 0 \conj{ for } k\geq 1 \label{k1b}
\end{align}
hold in $\NAlg(K, \bA)$.
Eqn. \eqref{k1b} agrees with the (R2)-relation in Cor. \ref{diskprescor}.
This is true for Eqn. \eqref{k1a} because Lem. 5.18 \cite{sillyman1} shows that this relation holds for consecutive arcs in the composition subalgebra of a disk.

The pairs $\{ \vp(z_S), \vp(z_{P_1}) \}$ and $\{ \vp(z_S), \vp(z_{Q_1}) \}$
are consecutive arcs in the disks $C_0$ and $D_0$ respectively. For $0< i
\leq m-2$ and $0 < j \leq n-2$, the arcs $\{ \vp(z_{P_i}), \vp(z_{P_{i+1}})
\} = \{E_i, E_{i+1} \}$ and $\{ \vp(z_{Q_j}), \vp(z_{Q_{j+1}}) \} = \{F_j,
F_{j+1}\}$ are consecutive arcs in the disks $C_{i-1}\sqcup_{P_{i+1}} C_i$
and $D_{j-1} \sqcup_{Q_{j+1}} D_j$ respectively. The arcs $\{ \vp(z_{P_{m-1}}), \vp(z_{T}) \} = \{E_{m-1}, N \}$ are consecutive in the disk $D$ from Def. \ref{def:newarcN}. Likewise the arcs $\{ \vp(z_{Q_{n-1}}), \vp(z_{T}) \} = \{F_{n-1}, N \}$ are consecutive in the disk $D'$ from Def. \ref{def:newarcN}.

{\it For Eqn. \eqref{k2:itm}:} Notice that when $z$ and $w$ are 1-dimensional simple modules $z\cdot w = 2$ implies that $z=w$. So suppose that $z$ is a 1-dimensional simple module and $\vp(z)$ is an arc in $K_{m,n}$. Then since the relation
  $$[\vp(z), \s^k\vp(z)]_{1} = 0$$
agrees with the (R1)-relation in Cor. \ref{diskprescor}, it must hold in $\NAlg(K,\bA)$ when $\vp(z)$ is the boundary arc of a disk $R \hookrightarrow K$ that satisfies the embedding criteria in Prop. \ref{nembthm}. The disk $R$ which implies this relation depends on the generator $z$. For the generators $z_S$, $z_{P_i}$ and $z_{Q_j}$, the arcs $\vp(z_S) = S$, $\vp(z_{P_i}) = E_i$ and $\vp(z_{Q_j})$ bound the disks $C_0$, $C_{i-1}$ and $D_{j-1}$ in $K$. Also for $z_T$, the arc $\vp(z_T) = N$ bounds the disk $D \subset K$ from Def. \ref{def:newarcN}.

{\it Step \#3.} The map $\psi$ is a left inverse of $\varphi$, $\varphi\circ\psi = 1_{\NAlg(K,\bA)}$. It suffices to check this equation on generators. The definitions in Step \#1 and Step \#2 show that the equation holds tautologically on the generators $S$, $E_i$ for $1\leq i < m$ and $F_j$ for $1 \leq j < n$.

Now we show that relation holds on generators $P_i$ using induction. The base case follows by definition because $P_0 := S$. Assuming $\varphi\psi(P_i) = P_i$,
\begin{align*}
  \varphi\psi(P_{i+1}) &= \varphi[\psi(P_i), z_{P_i}]_q & \nt{\eqref{psi1eq}}\\
  &= [\varphi\psi(P_i), \varphi(z_{P_i})]_q & \normaltext{(Step \#2)}\\
  &= [P_i, \varphi(z_{P_i})]_q & \nt{(Induction)}\\
  &= [P_i, E_i]_q & \nt{\eqref{defofphi}}\\
  &= P_{i+1} & \nt{(Lem. \ref{ddplemma})}
\end{align*}
In summary, by symmetry 
\begin{equation}\label{invpieqn}
\varphi\psi(P_i) = P_i \nt{ for } 0 \leq i < m  \conj{ and } \varphi\psi(Q_j) = Q_j \nt{ for } 0\leq j < n.
  \end{equation}

It remains to show that $\varphi\psi = 1$ on the generators $E_m$, $F_n$ and
$T$.
\begin{align*}
  \varphi\psi(E_m) &= \varphi[\psi(Q_{n-1}), z_T]_q & \nt{\eqref{psi2eq}}\\
  &= [\varphi\psi(Q_{n-1}), \varphi(z_T)]_q & \nt{(Step \#2)}\\
  &= [Q_{n-1}, \varphi(z_T)]_q & \nt{\eqref{invpieqn}}\\
  &= [Q_{n-1}, N]_q & \nt{\eqref{defofphi}}\\
  &= E_m & \nt{(Lem. \ref{ddplemma})}
  \end{align*}
The proof that $\varphi\psi(F_n) = F_n$ is symmetric. Lastly,
\begin{align*}
  \varphi\psi(T) &= \varphi[\psi(P_{m-1}), [\psi(Q_{n-1}), z_T]_q ]_q & \nt{\eqref{psi3eq}}\\
  &= [\vp\psi(P_{m-1}), [\vp\psi(Q_{n-1}), \vp(z_T)]_q ]_q & \nt{(Step \#2)}\\
&= [P_{m-1}, [Q_{n-1}, \vp(z_T)]_q ]_q & \nt{\eqref{invpieqn}}\\
&= [P_{m-1}, [Q_{n-1}, N]_q ]_q & \nt{\eqref{defofphi}}\\
&= [P_{m-1}, E_m ]_q & \nt{(Lem. \ref{ddplemma})}\\
&= T & \nt{(Lem. \ref{ddplemma})}
\end{align*}

  \end{proof}

\section*{Glossary of notation}\label{kcglossarysec}
\noindent
\begin{multicols}{2}
\begin{list}{}{
  \renewcommand{\makelabel}[1]{#1\hfil}
}
\item[$X\cdot Y$] \S\ref{hallsec}
\item[$(X,Y)$] \S\ref{hallsec}
\item[$\sqcup_{i,j}$] \S\ref{gluingconjecturesec}
\item[$\aprod_{i,j}$] Thm. \ref{gluethm}
\item[$A$] \S\ref{surfsec}
\item[$\A_m$] Ex. \ref{exex}
\item[$\A_{m,n}$] \S\ref{arcansec}
\item[$\bA_{m,n}$] \S\ref{arcansec}
\item[balanced] Def. \ref{balanceddef}
\item[boundary path] Def. \ref{fukcatdef}
\item[$C_i$] Def. \ref{disklabelsdef}
\item[$D^2$] disk
\item[$D_i$] Def. \ref{disklabelsdef}
\item[$\DC(Q)$] Prop. \ref{kalgdef}
\item[$D^\pi\eC$] Def. \ref{splitcldef}
\item[$D^\pi\F(S,A)$] Def. \ref{scfukcatdef}
\item[$\tDHa(\eD)$] \S\ref{hallsec}
\item[degree] Def. \ref{fukcatdef}
\item[double pair] Def. \ref{balanceddef}
\item[$E_i$] \S\ref{nanparamsec}
\item[edge] \S\ref{graphdef}
\item[$\Phi$] \eqref{parameqn}
\item[$F_i$] \S\ref{nanparamsec}
\item[$\F(S,A)$] Def. \ref{fukcatdef}
\item[$\Alg(D^2,\A_3)$] Cor. \ref{diskprescor}
\item[$\Alg(S,A)$] Def. \ref{fcompalgdef}
\item[foliation data] Def. \ref{foldatadef}
\item[full arc system] \S \ref{surfsec}
\item[fully formal] Def. \ref{fullyformaldef}
\item[$\ga_A$] Thm. \ref{prop:naivemaps}
\item[$\Ga$] \S\ref{graphdef}
\item[$G_{m,n}$] \S\ref{nanparamsec}
\item[$\gr(S)$] \eqref{greqn}
\item[$K_{m,n}$] \S\ref{arcansec}
\item[$M$] \S\ref{surfsec}
\item[map of surfaces] \S\ref{surfsec}
\item[$\NAlg(S,A)$] Def. \ref{naivedef}
\item[$\Omega$] \S\ref{graphdef}
\item[$P_i$] \S\ref{nanparamsec}
\item[$Q_i$] \S\ref{nanparamsec}
\item[$\s$] suspension \S\ref{gradingsec}
\item[$S$] \S\ref{surfsec}
\item[$S$] arc in $\bA_{m,n}$
\item[$T$] arc in $\bA_{m,n}$
\item[$kV_{m,n}$] \eqref{kvmneqn}
\item[$V_{m,n}$] \S\ref{nanparamsec}
\item[vertex] \S\ref{graphdef}
\item[$z_X$] simple module \ref{kalgdef}
\end{list}
\end{multicols}

\subsubsection*{The value of $q$}
If $s := \vnp{\FF_s}$ then $q^2 = s$. This convention agrees with the paper \cite{sillyman1}.

\subsubsection*{$q$-Analogues of the Lie bracket}
If $A$ is a $\ZZ[q]$-algebra and $x,y\in A$ then the $q$-analogue of the Lie bracket $[x,y]_q\in A$ is defined by the equation below.
\[
  [x,y]_q := xy - qyx
\]

This $q$-commutator satisfies a number of elementary algebraic identities, several of which are listed at the end of \cite{sillyman1}.

\bibliography{clipfish}{}

\providecommand{\bysame}{\leavevmode\hbox to3em{\hrulefill}\thinspace}
\providecommand{\MR}{\relax\ifhmode\unskip\space\fi MR }
\providecommand{\MRhref}[2]{%
  \href{http://www.ams.org/mathscinet-getitem?mr=#1}{#2}
}
\providecommand{\href}[2]{#2}
\begin{thebibliography}{HKK17}

\bibitem[BS12]{BS}
Igor Burban and Olivier Schiffmann, \emph{On the {H}all algebra of an elliptic
  curve, {I}}, Duke Math. J. \textbf{161} (2012), no.~7, 1171--1231.

\bibitem[BS13]{BSWPL}
\bysame, \emph{The composition {H}all algebra of a weighted projective line},
  J. Reine Angew. Math. \textbf{679} (2013), 75--124. \MR{3065155}

\bibitem[CF94]{CraneFrenkel}
Louis Crane and Igor~B. Frenkel, \emph{Four-dimensional topological quantum
  field theory, {H}opf categories, and the canonical bases}, J. Math. Phys.
  \textbf{35} (1994), no.~10, 5136--5154.

\bibitem[CS]{sillyman2}
Benjamin Cooper and Peter Samuelson, \emph{{H}all algebras of surfaces {II}},
  In preparation.

\bibitem[CS18]{sillyman1}
\bysame, \emph{{H}all algebras of surfaces {I}}, Journal of the Institute of
  Mathematics of Jussieu (2018).

\bibitem[Haia]{Haiden2}
Fabian Haiden, \emph{Flags and tangles}, arXiv:1910.04182.

\bibitem[Haib]{Haiden1}
\bysame, \emph{Legendrian skein algebras and hall algebras}, arXiv:1908.10358.

\bibitem[HKK17]{HKK}
Fabian Haiden, Ludmil Katzarkov, and Maxim Kontsevich, \emph{Flat surfaces and
  stability structures}, Publ. Math. Inst. Hautes \'{E}tudes Sci. \textbf{126}
  (2017), 247--318.

\bibitem[HL15]{HL}
David Hernandez and Bernard Leclerc, \emph{Quantum {G}rothendieck rings and
  derived {H}all algebras}, J. reine angew. Math. \textbf{701} (2015), 77--126.

\bibitem[Igu04]{Igusa}
Kiyoshi Igusa, \emph{Graph cohomology and {K}ontsevich cycles}, Topology
  \textbf{43} (2004), no.~6, 1469--1510.

\bibitem[LH03]{LH}
Kenji Lef\`{e}vre-Hasegawa, \emph{Sur les {A}-infini cat\'{e}gories}, Ph.D.
  thesis, Univ. Paris 7, 2003.

\bibitem[LP18a]{LP}
Yank\i Lekili and Alexander Polishchuk, \emph{Auslander orders over nodal
  stacky curves and partially wrapped {F}ukaya categories}, J. Topol.
  \textbf{11} (2018), no.~3, 615--644. \MR{3830878}

\bibitem[LP18b]{Lekili2}
Yanki Lekili and Alexander Polishchuk, \emph{Derived equivalences of gentle
  algebras via {F}ukaya categories}, arXiv:1801.06370v3 (2018).

\bibitem[Lus91]{Lusztig}
George Lusztig, \emph{Quivers, perverse sheaves, and quantized enveloping
  algebras}, J. Amer. Math. Soc. \textbf{4} (1991), no.~2, 365--421.

\bibitem[MS17]{MortonSamuelson}
Hugh Morton and Peter Samuelson, \emph{The {HOMFLYPT} skein algebra of the
  torus and the elliptic {H}all algebra}, Duke Math. J. \textbf{166} (2017),
  no.~5, 801--854.

\bibitem[Nee01]{Neeman}
Amnon Neeman, \emph{Triangulated categories}, Annals of Mathematics Studies,
  vol. 148, Princeton University Press, Princeton, NJ, 2001.

\bibitem[Rin90]{Ringel}
Claus~M. Ringel, \emph{Hall algebras and quantum groups}, Invent. Math.
  \textbf{101} (1990), 583--592.

\bibitem[Sch14]{Schiffler}
Raff Schiffler, \emph{Quiver representations}, Canadian Mathematical Society,
  2014.

\bibitem[SV11]{SV11}
O.~Schiffmann and E.~Vasserot, \emph{The elliptic {H}all algebra, {C}herednik
  {H}ecke algebras and {M}acdonald polynomials}, Compos. Math. \textbf{147}
  (2011), no.~1, 188--234.

\bibitem[To{\"e}06]{Toen}
Bertrand To{\"e}n, \emph{Derived {H}all algebras}, Duke Math. J. \textbf{135}
  (2006), no.~3, 587--615.

\bibitem[XX08]{XX}
Jie Xiao and Fan Xu, \emph{Hall algebras associated to triangulated
  categories}, Duke Math. J. \textbf{143} (2008), no.~2, 357--373.

\end{thebibliography}
\bibliographystyle{amsalpha}

\end{document}